\documentclass[12pt]{amsart}

\usepackage{amsmath, amssymb, amsthm, mathrsfs, graphicx, hyperref,float}
\usepackage[shortlabels]{enumitem}
\usepackage[all,cmtip]{xy}
\usepackage[numbers]{natbib}
\usepackage[margin=1in]{geometry}

\newtheorem{theorem}{Theorem}
\newtheorem{proposition}[theorem]{Proposition}
\newtheorem{lemma}[theorem]{Lemma}
\newtheorem{corollary}[theorem]{Corollary}

\theoremstyle{definition}
\newtheorem{definition}[theorem]{Definition}
\newtheorem{example}[theorem]{Example}
\newtheorem{remark}[theorem]{Remark}

\numberwithin{theorem}{section}

\makeatletter
\providecommand{\leftsquigarrow}{%
  \mathrel{\mathpalette\reflect@squig\relax}%
}
\newcommand{\reflect@squig}[2]{%
  \reflectbox{$\m@th#1\rightsquigarrow$}%
}
\makeatother

\newenvironment{situation}{
   \bigskip

  \noindent \textbf{Standard Situation:} }
{
  \bigskip
}

\newcommand{\Spec}{\mathrm{Spec}\,}

\newcommand{\pp}{\mathfrak{p}}
\newcommand{\qq}{\mathfrak{q}}

\newcommand{\N}{\mathbb{N}}

\newcommand{\Z}{\mathbb{Z}}

\newcommand{\ol}[1]{\overline{#1}}

\newcommand{\core}{\mathrm{core}}
\newcommand{\Div}{\mathrm{Div}}

\newcommand{\colim}{\mathop{\mathrm{colim}}}
\newcommand{\Cech}{\v{C}ech }

\newcommand{\Tor}{\mathrm{Tor}}

\newcommand{\trop}{\mathrm{trop}}

\title{Contractions of subcurves of families of log curves}
\author{Sebastian Bozlee}
\email{Sebastian.Bozlee@tufts.edu}

\begin{document}

\begin{abstract}
Let $C$ be a nodal curve, and let $E$ be a union of semistable subcurves of $C$. We consider the problem of contracting the connected components of $E$ to singularities in a way that preserves the genus of $C$ and makes sense in families, so that this contraction may induce maps between moduli
spaces of curves.

In order to do this, we introduce the notion of mesa curve, a nodal curve $C$ with a logarithmic structure and a piecewise linear function $\ol{\lambda}$ on the tropicalization of $C$. This piecewise linear function determines a subcurve $E$. We then construct a contraction of $E$ inside of $C$ for families of mesa curves. Resulting singularities include the elliptic Gorenstein singularities.
\end{abstract}

\maketitle

\tableofcontents

\section{Introduction}

This paper constructs a general contraction map for families of log curves, generalizing contraction constructions found in \cite{rsw}, \cite{keli_thesis}, and \cite{hassett_hyeon}. Similar contraction constructions have been used to construct modular desingularizations of moduli of stable maps in genus one (see \cite{keli_thesis}, \cite{rsw}) and genus two (see \cite{luca_francesca}), as well as constructing spaces appearing in the Hassett-Keel program for $\ol{M}_{g,n}$.

Let $\tau : C \to \ol{C}$ be a map of proper algebraic curves, and let $E$ be the union of the irreducible components of $C$ on which $\tau$ is constant. We say that $\tau$ is a \emph{contraction of $E$ inside of $C$} if
\begin{enumerate}
  \item $\tau$ is surjective with connected fibers,
  \item $\tau$ restricts to an isomorphism on $C - E$, and
  \item for each connected component $Z$ of $E$, we have that $\tau(Z)$ is a singularity with number of branches equal to $|Z \cap \ol{C - Z}|$
  and with genus equal to the arithmetic genus of $Z$.
\end{enumerate}

Here, the genus of a singular point $p$ is $\delta(p) - m(p) + 1$, where $\delta(p)$ is the $\delta$-invariant of $p$ (see \cite[Chapter IV, Exercise 1.8]{hartshorne}) and $m(p)$ is its number of branches.

A \emph{family of curves} is a morphism $\pi : C \to S$ of algebraic spaces so that $\pi$ is flat and proper with connected, reduced, 1-dimensional geometric fibers. We call a closed subscheme $E \subseteq C$ a \emph{family of subcurves} if the geometric fibers of $E \to S$ are unions of connected components of the geometric fibers of $C \to S$. If $\pi : C \to S$ is a family of curves and $E \subseteq C$ is a family of subcurves,
we say that a diagram
\[
\xymatrix{
  C \ar[r]^{\tau} \ar[rd]_{\pi} & \ol{C} \ar[d]^{\ol{\pi}} \\
 & S
}
\]
is a \emph{contraction of $E$ inside of $C$} if
\begin{enumerate}
  \item $\ol{\pi} : \ol{C} \to S$ is a family of curves,
  \item $\tau$ restricts to an isomorphism on $C - E$, and
  \item for each geometric point $\ol{s}$ of $S$, the restriction $\tau_{\ol{s}} : C \times_S \ol{s} \to \ol{C} \times_S \ol{s}$ is a contraction of $E \times_S \ol{s}$ inside of $C \times_S \ol{s}$.
\end{enumerate}
We also say that $\ol{C}$ is a contraction of $C$.

\medskip

Contractions arise frequently in the study of moduli spaces of curves. A basic reason for their ubiquity is the semistable reduction theorem:
informally, any limit of a family of smooth curves is a contraction of a nodal curve. Moreover, given any two limits of a family of smooth curves, both are contractions of a common nodal curve. 

Of particular interest to us are contractions which induce morphisms between moduli spaces of curves.
Suppose $\ol{\mathcal{M}}_1$ and $\ol{\mathcal{M}}_2$ are moduli spaces of curves, that is, algebraic stacks whose
$S$-points are identified with certain families of curves $\pi : C \to S$ possibly equipped with some extra data.
To define a morphism $\ol{\mathcal{M}}_1 \to \ol{\mathcal{M}}_2$, one first constructs a contraction
\[
\xymatrix{
  C \ar[r]^{\tau} \ar[rd]_{\pi} & \ol{C} \ar[d]^{\ol{\pi}} \\
 & S
}
\]
for each family of curves $\pi : C \to S$ in $\ol{\mathcal{M}}_1(S)$, so that $\ol{\pi} : \ol{C} \to S$ lies in $\ol{\mathcal{M}}_2(S)$. If these contractions are compatible with pullback, we obtain a morphism $\ol{\mathcal{M}}_1 \to \ol{\mathcal{M}}_2$ by taking each family $C \in \ol{\mathcal{M}}_1(S)$ to its contraction $\ol{C} \in \ol{\mathcal{M}}_2(S)$.

Equivalently, one may construct a contraction
\[
\xymatrix{
  \mathcal{C} \ar[r]^{\tau} \ar[rd]_{\pi} & \ol{\mathcal{C}} \ar[d]^{\ol{\pi}} \\
 & \ol{\mathcal{M}}_1
}
\]
where $\mathcal{C} \to \ol{\mathcal{M}}_1$ is the universal curve and $\ol{\mathcal{C}} \to \ol{\mathcal{M}}_1$ is a family of curves defining a map to $\ol{\mathcal{M}}_1 \to \ol{\mathcal{M}}_2$.

For example, in \cite{hassett}, contraction maps give morphisms between the moduli spaces of weighted-stable marked curves given by different weightings of marked points.
In \cite{hassett_hyeon}, a contraction map is used to construct a morphism from the Deligne-Mumford space of curves to the space of pseudostable curves.
In Keli Santos-Parker's thesis \cite{keli_thesis} (later published in \cite{rsw}), this procedure is used with an extra step. To resolve the indeterminacy
of the rational maps between the Deligne-Mumford space $\ol{\mathcal{M}}_{1,n}$ and Smyth's spaces
$\ol{\mathcal{M}}_{1,n}(m)$ of $m$-stable  curves (see \cite{smyth_mstable}), Santos-Parker constructs two maps
\[
  \ol{\mathcal{M}}_{1,n} \leftarrow \ol{\mathcal{M}}_{1,n}^{rad} \to \ol{\mathcal{M}}_{1,n}(m).
\]
The first map is a blowup of $\ol{\mathcal{M}}_{1,n}$. Logarithmic geometry gives this blowup a modular interpretation: it parametrizes
so-called \emph{radially-aligned log curves}. Using one more piece of logarithmic data for each $m$, Santos-Parker is able to construct a contraction
of (a slight modification of) the universal curve of $\ol{\mathcal{M}}_{1,n}^{rad}$ to a family of $m$-stable curves, which yields the second map.

\medskip

For each morphism of moduli spaces mentioned above, the contraction was constructed by choosing a line bundle $\mathcal{L}$
on the universal curve $\mathcal{C}$, then using the contraction
\[
\xymatrix{
  \mathcal{C} \ar[rr]^-{\tau} \ar[rd]_{\pi} & & \ol{\mathcal{C}} := \underline{\mathrm{Proj}}_S \bigoplus_{k = 0}^\infty \pi_*(\mathcal{L}^{\otimes k}) \ar[ld]^{\ol{\pi}} \\
  & \ol{\mathcal{M}}_1 &
}
\]

This construction is elegant, but there is a frustrating point: in general, $\pi_*$ does not commute with base change, so one cannot see the image of an individual curve $C$
by computing $\mathrm{Proj} \bigoplus_k \Gamma(C, \mathcal{L}|_C)$. Unless you get lucky, it is necessary to first compute a sufficiently
nice smoothing family of $C$, compute the contraction of this family, and then restrict back to $C$.

\medskip

Our goal in this paper is to construct a contraction for families of log curves so that the construction of the contraction commutes with base change.
Our construction will be capable of contracting multiple subcurves of any genus inside of a family of any genus. The singularities that result include the elliptic Gorenstein singularities, but
others occur as well.

The move that makes this possible is a log geometric encoding of the contraction data in what we will call a \emph{mesa curve}, consisting of a
log curve $\pi : C \to S$ of any genus and a section $\ol{\lambda}$ of the characteristic sheaf of $C$ which we call a mesa. We give a formal definition below in
Definition \ref{def:mesa_curve}, and only sketch its flavor here.

A log curve $C$ has an underlying pointed nodal curve (see Section \ref{sec:logcurves} below). Recall that the dual graph of a pointed nodal curve $C$
has a vertex for each component, an edge for each node, and a ``leg" or ``half-edge" for each marked point of $C$. A log curve $C$
carries additional information that we use as a ``length" for each edge of its dual graph. We call the dual graph together with these lengths
the tropicalization of $C$. The datum $\ol{\lambda}$ is identified with a piecewise linear function on the tropicalization of $C$, see Figure \ref{fig:mesa_example_early}.

\begin{figure}
\begin{center}
\includegraphics[width=\textwidth]{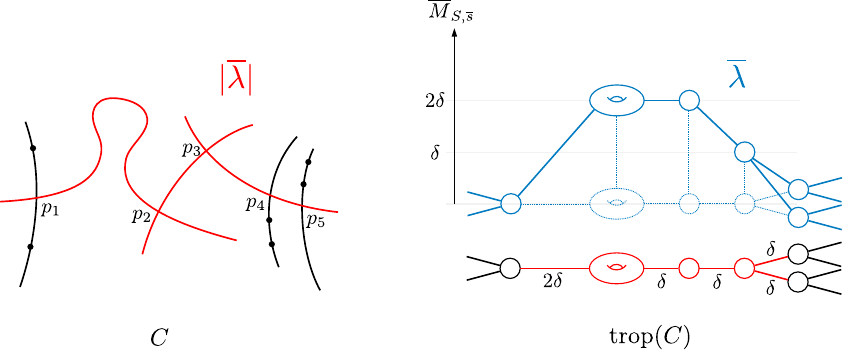}
\end{center}
\caption{A log curve $C$, a section $\ol{\lambda}$ of its characteristic sheaf, and the support $|\ol{\lambda}|$ of that section.}
\label{fig:mesa_example_early}
\end{figure}

We require that this piecewise linear function $\ol{\lambda}$ has a particular shape. (Hence the name ``mesa.") The locus where $\ol{\lambda}$
is nonzero determines a subcurve $E = |\ol{\lambda}|$ which we call the support of $\ol{\lambda}$. There is also an associated invertible sheaf
$\mathscr{O}_C(-\ol{\lambda})$ and we require that $H^1(E, \mathscr{O}_C(-\ol{\lambda}) \otimes \mathscr{O}_E) = 0$.

Our main theorem is that mesa curves admit a contraction map collapsing $E$ compatibly with base change and resulting in a flat family of singular curves.

\begin{theorem} \label{thm:contraction}
Let $S$ be an fs log scheme, $(\pi : C \to S, \ol{\lambda})$ a mesa curve, and $E = |\ol{\lambda}|$. Then there is a contraction of $E$ inside of
$C$
\[
   \xymatrix{
   C \ar[r]^\tau \ar[dr]_\pi & \ol{C} \ar[d]^{\ol{\pi}} \\
   & S
  }
\]
so that
\begin{enumerate}
  \item if a connected component of $E \times_S {\ol{s}}$ is the support of a steep mesa of genus one, its image is an elliptic Gorenstein singularity;
  \item for each morphism of schemes $T \to S$, there is a natural isomorphism $\ol{C \times_S T} \cong \ol{C} \times_S T$, compatible with the maps from $C$ and to $S$.
\end{enumerate}
\end{theorem}
We will define ``steep mesa" following Definition \ref{def:simple_mesa}. The construction of the contracted curve is given locally on $S$ in Definition \ref{def:contraction}.

Our method of proof is to explicitly construct a structure sheaf for $\ol{C}$, then verify that it has the desired properties. This requires a fair amount of technical work compared to constructing a line bundle, but we gain commutativity with base change and a relatively explicit description of the image of a particular curve, see Proposition \ref{prop:singularity_has_correct_genus}.

We pause for a moment to describe the resulting ring of functions independently of log structures: let $C$ be a nodal curve and $E$ a connected semistable subcurve of genus $g$. Let $Z$ be the union of the irreducible components of $C$
not contained in $E$, and let $x_1, \ldots, x_m$ be local parameters on $Z$ of the points $p_1, \ldots, p_m$ where $Z$ meets $E$. Then for any log structure and choice of $\ol{\lambda}$ making $C$ into a mesa curve with $E = |\ol{\lambda}|$,
there is an open neighborhood $U$ of $\tau(E)$ in $\ol{C}$ so that
\[
  \Gamma(U, \mathscr{O}_{\ol{C}}) = \bigg\{ f \in \mathscr{O}_Z(U \cap Z) \, \, \bigg| \, \, f(p_i) = f(p_j) \text{ for all }i,j \text{ and } \left[ \frac{\partial f}{\partial x_i}(p_i) \right]_{i = 1}^m \in V \bigg\}
\]
where $V$ is a linear subspace of $k^m$ of codimension $g$. We give several examples in Section \ref{ssec:reg_functions}. These examples show that possible singularities of $\ol{C}$ include all of the elliptic Gorenstein singularities (such as cusps and tacnodes), elliptic Gorenstein singularities meeting a rational $n$-fold point transversally, and
a union of cusps meeting transversally.

\subsection{Future directions}

A potential application of this result is to the construction and classification of alternative modular compactifications of $\mathcal{M}_{g,n}$. To accomplish this,
we consider logarithmic modifications
\[
  \ol{\mathcal{M}}_{g,n} \leftarrow \mathcal{M}
\]
of the Deligne-Mumford stack of curves, and then identify the mesas on the universal curve of $\mathcal{M}$.
The problem of identifying these mesas is essentially combinatorial. By iterating through them, we iterate through many contractions of the universal curve.
Each such contraction yields a morphism from $\mathcal{M}$ to the ``moduli space of all smoothable curves" $\mathcal{V}_{g,n}$
\cite[Appendix B]{smyth_zstable}. The images of these morphisms are candidates for alternative modular compactifications of $\mathcal{M}_{g,n}$ in the sense of \cite{smyth_zstable}.

The author began an exploration of the possibilities in his thesis. Much remains to be done, but already in forthcoming work we expect
to show that this technique identifies new semistable compactifications of $\mathcal{M}_{1,n}$. These new compactifications appear to exhaust the semistable modular
compactifications of $\mathcal{M}_{1,n}$ with reduced Gorenstein curves and distinct markings.

In \cite{luca}, Battistella constructs modular compactifications of $\mathcal{M}_{2,n}$ by semistable Gorenstein curves in the spirit of Smyth's
$m$-stable spaces. It would be tempting to imitate Santos-Parker and produce a resolution of the rational map between $\ol{\mathcal{M}}_{2,n}$
and Battistella's spaces. Regrettably, our construction does not produce contractions to the genus two Gorenstein singularities considered in
\cite{luca}: the topological types of the semistable models of these singularities are precisely those that fail to satisfy the
$H^1(E, \mathscr{O}_E(-\ol{\lambda})) = 0$ condition. However, the divisor $D$ considered by Battistella in this context strongly resembles our
$\ol{\lambda}$. It is possible that a more refined formula for $\ol{C}$ can be found which includes contractions to Gorenstein singularities of
genus two.

Another possible application is the construction of modular desingularizations of spaces of stable maps.
Santos-Parker's contraction construction was applied in \cite{rsw} to give a modular interpretation of the Vakil-Zinger desingularization
of the space of genus stable maps originally constructed in \cite{vakil_zinger}. In genus two, Hu, Li, and Niu have constructed a desingularization
of the space of stable maps in the spirit of \cite{vakil_zinger}. Battistella and Carrocci show in the preprint \cite{luca_francesca}
that a moduli problem similar to that of Santos-Parker yields a modular desingularization of the space of stable genus two maps. With more
sophistication, contractions informed by log geometry may yield modular desingularizations of stable maps in yet higher genus.

\section*{Acknowledgements}
The content of this article forms the first part of my doctoral thesis at the University of Colorado at Boulder. I would like to thank my advisor, Jonathan Wise, for his inexhaustible patience and guidance, without which this work would not be possible. I would also like to acknowledge Dan Abramovich, Luca Battistella, Renzo Cavalieri, Francesca Carocci, Qile Chen, Andy Fry, Brendan Hassett, Leo Herr, Keli Parker, Dhruv Ranganathan, Hanson Smith, David Smyth, and John Willis for their interest, encouragement, and many helpful conversations.

\section{Preliminaries on Logarithmic Geometry}

\subsection{Log structures}

For a more comprehensive introduction to log structures, the reader may wish to consult \cite{kato_log_structures} or \cite{ogus}.
All monoids are assumed commutative. For an algebraic space $X$, regard $\mathscr{O}_X$ as a sheaf of monoids on the small \'etale site of $X$ with the operation induced by the multiplication of $\mathscr{O}_X$.

\begin{definition}
 A \emph{log structure} on an algebraic space $X$ consists of
\begin{enumerate}
  \item an \'etale sheaf of monoids $M_X$ on $X$ and
  \item a morphism of sheaves of monoids $\epsilon : M_X \to \mathscr{O}_X$
\end{enumerate}
so that $\epsilon$ restricts to an isomorphism $\epsilon : \epsilon^{-1}(\mathscr{O}_X^*) \to \mathscr{O}_X^*$.

The pair  $(X, \epsilon : M_X \to \mathscr{O}_X)$ is called a \emph{log algebraic space}. We will often abbreviate $(X, \epsilon: M_X \to \mathscr{O}_X)$ to $(X, M_X)$ or just $X$. It is convenient to identify $\epsilon^{-1}(\mathscr{O}_X^*)$ with
$\mathscr{O}_X^*$. The quotient sheaf of monoids $\ol{M}_X := M_X / \epsilon^{-1}(\mathscr{O}_X^*)$ is called the \emph{characteristic sheaf} of $X$.
\end{definition}

Throughout this section, when we take a stalk of an \'etale sheaf at a geometric point, we intend the \'etale notion of stalk.

Given log algebraic spaces $(X, \epsilon : M_X \to \mathscr{O}_X)$ and $(Y, \eta: M_Y \to \mathscr{O}_Y)$, a \emph{morphism of log algebraic spaces} from $X$ to $Y$ is a morphism of algebraic spaces $f : X \to Y$ and a morphism $f^\flat : f^{-1}M_Y \to M_X$
so that
\[
  \xymatrix{
   f^{-1}M_Y \ar[r]^{f^\flat} \ar[d]_{f^{-1}\eta} & M_X \ar[d]^{\epsilon} \\
   f^{-1}\mathscr{O}_Y \ar[r]_{f^\sharp} & \mathscr{O}_X
  }
\]
commutes.

It is convenient to have a way to specify log structures without explicit reference to $\mathscr{O}_X^*$. If $e : M \to \mathscr{O}_X$ is any morphism from a sheaf of monoids $M$ to
$\mathscr{O}_X$, we may define the \emph{associated log structure} to be the dashed arrow in the pushout diagram of \'etale sheaves of monoids on $X$ below:
\[
  \xymatrix{
    e^{-1}(\mathscr{O}_X^*) \ar[r] \ar[d]_{e} & M \ar[ddr]^e \ar[d] \\
    \mathscr{O}_X^* \ar[drr] \ar[r] & M^a \ar@{-->}[dr] \\
   & & \mathscr{O}_X
  }
\]

Given a morphism of algebraic spaces $f : X \to Y$ and a log structure $\eta : M_Y \to Y$, the \emph{inverse image log structure} on $X$, denoted $f^*M_Y$, is the log structure associated to the composition $f^{-1}M_Y \to f^{-1}\mathscr{O}_Y \to \mathscr{O}_X$. If $X$ has a log structure and $f : X \to Y$ is a morphism of log algebraic spaces, then there is an induced morphism of sheaves $f^*M_Y \to M_X$. If this morphism is an isomorphism, we say $f$ is \emph{strict}.

A \emph{chart} for a log structure $\epsilon : M \to \mathscr{O}_X$ is a morphism of monoids $e : P \to \Gamma(X, \mathscr{O}_X)$ so that when we take the induced morphism $e : \underline{P} \to \mathscr{O}_X$ from a constant sheaf of
monoids, then the associated log structure of this morphism, we arrive at a log structure isomorphic to $\epsilon : M \to \mathscr{O}_X$. We think of a chart as a choice of global generators of the log structure. A log structure is said to be
\emph{quasicoherent} if it admits a chart \'etale locally.

For each monoid $P$, we may give $\Spec \Z[P]$ a log structure associated to the usual map $P \to \Z[P]$. A chart $e : P \to \Gamma(X, \mathscr{O}_X)$ is then equivalent to a strict map of algebraic spaces $X \to \Spec \Z[P]$. Analogously, a \emph{chart of a morphism} $f : X \to Y$
of log algebraic spaces is a commutative square
\[
  \xymatrix{
    X \ar[d]_f \ar[r] & \Spec \Z[Q] \ar[d] \\
    Y \ar[r] & \Spec \Z[P]
  }
\]
with strict horizontal arrows where the right vertical map is induced by a morphism $P \to Q$ of monoids.

A monoid $P$ is called \emph{finitely generated} if there is a surjective morphism $\N^r \to P$ for some integer $r$. There is a groupification functor $P \mapsto P^{gp}$ from monoids to groups, computed by formally adding inverses to all elements
of $P$. A monoid $P$ is said to be \emph{integral} if the natural morphism $P \to P^{gp}$ is injective. If $P$ is finitely generated and integral, we say $P$ is \emph{fine}. An integral monoid $P$ is \emph{saturated} if, for any $x \in P^{gp}$ and $n \in \N$ so that $n \cdot x \in P$, we have $x \in P$. A monoid that is both fine and saturated is called \emph{fs}. A log structure $\epsilon : M_X \to \mathscr{O}_X$
is said to be \emph{coherent} (resp. \emph{integral}, \emph{fine}, \emph{saturated}, \emph{fs}) if it locally has a chart by finitely generated (resp. integral, fine, saturated, fs) monoids.

A morphism $i : T' \to T$ of log algebraic spaces is called an \emph{exact closed immersion} if it is strict and the underlying map of algebraic spaces is a square-zero extension. A morphism $f : X \to Y$ of fine log algebraic spaces is called
\emph{log smooth} if
\begin{enumerate}
  \item The underlying map of algebraic spaces is locally of finite presentation;
  \item For any diagram of solid arrows
  \[
    \xymatrix{
      T' \ar[r] \ar[d]_{i} & X \ar[d]^f \\
      T \ar[r]  \ar@{-->}[ru] & Y
    }
  \]
  where $i : T' \to T$ is an exact closed immersion of affine log schemes, there is a morphism of log algebraic spaces completing the diagram.
\end{enumerate}
The underlying morphism of algebraic spaces of a log smooth map need not be smooth.

A homomorphism of integral monoids $h : Q \to P$ is \emph{integral} if the induced map of rings $\Z[Q] \to \Z[P]$ is flat (cf. \cite[Proposition (4.1)]{kato_log_structures}). A morphism $f : X \to Y$ of integral log algebraic spaces is \emph{integral} if, for
any $x \in X$, the induced morphism $\ol{M}_{Y, \ol{f(x)}} \to \ol{M}_{X, \ol{x}}$ is integral.

\subsection{Log curves} \label{sec:logcurves}

\begin{definition} (cf. \cite[Definition 1.2]{fkato_deformations}) Let $S$ be an fs log scheme. A \emph{log curve} over $S$ is a log smooth and integral morphism $\pi : C \to S$ of fs log algebraic spaces such that every geometric fiber of $\pi$ is
a reduced and connected curve.
\end{definition}

In other words, a log curve $\pi : C \to S$ is a relative curve so that $\pi$ is locally of finite presentation on rings (from log smoothness) and on monoids (from the fs hypothesis), flat on rings (implied by log smoothness and integrality) and on monoids (in the sense of
integrality), and that satisfies a log version of the infinitesimal lifting criterion for smoothness.

F. Kato has shown that the underlying algebraic space of a log curve is nodal \cite[Theorem 1.3]{fkato_deformations}, and the log structure of a log curve is constrained in such a way that it encodes data very similar to that of a pointed nodal curve. We adapt his statement in \cite[Subsection 1.8]{fkato_deformations} for our purposes.

\begin{theorem} \label{thm:log_curve_local_structure}
Let $\pi : C \to S$ be a family of proper log curves. If $\ol{x} \in C$ is a geometric point with image $\ol{s} \in S$, let $Q = \ol{M}_{C, \ol{x}}$ and $P = \ol{M}_{S, \ol{s}}$. We may find connected strict \'etale neighborhoods
$V$ of $\ol{x}$ and $U$ of $\ol{s}$ fitting into a diagram
\[
  \xymatrix{
    C \ar[d]_{\pi} & V \ar[l]_{\text{\'et}} \ar[d] \ar[r]^{\text{\'et}} & V' \ar[r] \ar[d] & \Spec \Z[Q] \ar[d] \\
    S & U \ar[l]_{\text{\'et}} \ar@{=}[r] & U \ar[r] & \Spec \Z[P]
  }
\]
in which all horizontal arrows are strict and the labelled arrows are \'etale. There are three possibilities for the right square, depending on $Q$:
\begin{enumerate}
  \item (the smooth germ) $Q \cong P$, and the right square is isomorphic to
\[
   \xymatrix{
    \Spec \mathscr{O}_U[x] \ar[d] \ar[r] & \Spec \Z[P] \ar[d] \\
    U \ar[r] & \Spec \Z[P]
  }
\]
  where the right vertical arrow is induced by the identity map on $P$
  \item (the germ of a marked point) $Q \cong P \oplus \N$, and the right square is isomorphic to
   \[
   \xymatrix{
    \Spec \mathscr{O}_U[x] \ar[d] \ar[r] & \Spec \Z[P \oplus \N] \ar[d] \\
    U \ar[r] & \Spec \Z[P]
  }
  \]
  where the right vertical arrow is induced by the inclusion $P \to P \oplus \N$, and the top map sends the generator of $\N$ to $x$.
  \item (the node) There is an element $\delta_{\ol{x}} \in P$ mapping to an element $t$ of $\Gamma(U, \mathscr{O}_U)$, the monoid
   $Q$ is isomorphic to $\{ (p, p') \in P \oplus P \mid p - p' \in \Z\delta_{\ol{x}} \}$,
   and the right square is isomorphic to 
   \[
   \xymatrix{
     \Spec \mathscr{O}_U[x,y]/(xy - t)  \ar[d] \ar[r] & \Spec \Z[Q] \ar[d] \\
     U \ar[r] & \Spec \Z[P].
   }
   \]
 The right vertical arrow is induced by the diagonal map $P \to Q$, and
  the top horizontal arrow sends $(\delta_{\ol{x}}, 0)$ to $x$ and $(0, \delta_{\ol{x}})$ to $y$.
\end{enumerate}
\end{theorem}
\begin{proof}
This follows from \cite[Subsection 1.8]{fkato_deformations} by standard arguments.
\end{proof}

We will say that a geometric point $\ol{x} \in C$ is \emph{smooth}, \emph{marked}, or a \emph{node} according to the kind
of neighborhood it has.
If $\ol{x}$ is marked, the stalk of $\ol{M}_C$ at $\ol{x}$ is naturally isomorphic over $\ol{M}_{S,\pi(\ol{x})}$ to $\ol{M}_{S,\pi(\ol{x})} \oplus \N$; given a global section $\sigma$ of $\ol{M}_C$, we define the \emph{slope of $\sigma$ at $\ol{x}$} to be the
value of $\sigma_x$ in the $\N$ coordinate. If $\ol{x}$ is nodal, the element $\delta_{\ol{x}}$ of $\ol{M}_{S,\ol{s}}$ is called the \emph{smoothing parameter} of the node $\ol{x}$. We note that the monoid $Q$ in the case of the germ of a node
may also be presented as $\N\alpha \oplus \N\beta \oplus P/(\alpha + \beta \sim \delta_{\ol{x}})$, with $\alpha$ mapping to $x$ and $\beta$ mapping to $y$.

\medskip

We now introduce tropicalizations of log curves using the framework of \cite{ccuw}. Intuitively, a tropical curve is a graph whose edges are labelled with ``lengths" coming from a monoid, along with some edges of infinite length which are attached to a vertex at only one end. More formally:

\begin{definition}
A \emph{tropical curve} $\Gamma$ \emph{with edge lengths in a sharp monoid} $P$ consists of:
\begin{enumerate}
  \item A finite set $X(\Gamma) = V(\Gamma) \sqcup F(\Gamma)$. The elements of $V(\Gamma)$ are called the \emph{vertices} of $\Gamma$ and the elements of $F(\Gamma)$ are called the \emph{flags} of $\Gamma$.
  \item A \emph{root map} $r_\Gamma : X(\Gamma) \to X(\Gamma)$ which is idempotent with image $V(\Gamma)$.
  \item An involution $\iota_\Gamma : X(\Gamma) \to X(\Gamma)$ whose fixed point set contains $V(\Gamma)$. The subsets $\{ f, \iota_\Gamma(f) \}$ of $F(\Gamma)$
  of size two are called \emph{edges}, and the set of all edges is denoted $E(\Gamma)$. The subsets $\{ f, \iota_\Gamma(f) \}$ of $F(\Gamma)$ of size one are called \emph{half-edges},
  and the set of all half-edges is denoted $H(\Gamma)$.
  \item A function $g : V(\Gamma) \to \N$. Given a vertex $v$, $g(v)$ is called the \emph{genus of $v$}.
  \item A function $\delta : E(\Gamma) \to P$. Given an edge $e$, $\delta(e)$ is called the \emph{length of $e$}.
\end{enumerate}
\end{definition}

We imagine that each flag $f$ is half of an edge starting at the vertex $r_\Gamma(f)$. Given an edge $e = \{ f, \iota_\Gamma(f) \}$, we say the vertices $r_\Gamma(f)$ and $r_\Gamma(\iota_\Gamma(f))$ are \emph{incident} to $e$. The \emph{genus} of a tropical curve $\Gamma$ is $g(\Gamma) = b_1(\Gamma) + \sum_{v \in V(\Gamma)} g(v)$,
where $b_1(\Gamma)$ is the first Betti number of $\Gamma$, that is, $|E(\Gamma)| - |V(\Gamma)| + n$ where $n$ is the number of connected components of $\Gamma$.

The tropicalization of a log curve is the dual graph of its underlying nodal curve, enriched with the data of the smoothing parameters from its log structure.

\begin{definition}
Given a log curve $\pi: C \to S$ where $S$ is a log point, the \emph{tropicalization} $\trop(C)$ of $C$ is the tropical curve with
edge lengths in $\Gamma(S, \ol{M}_{S})$ which has
\begin{enumerate}
  \item a vertex for each component of $C$;
  \item an edge for each node of $C$, incident to the components of $C$ which form the branches of the node;
  \item a half-edge for each marked point of $C$, rooted at the component of $C$ to which the marked point belongs
\end{enumerate}
and
\begin{enumerate}
  \item for each vertex $v$, the genus $g(v)$ is the genus of the normalization of the corresponding component of $C$;
  \item for each edge $e$, the length $\delta(e) \in \Gamma(S, \ol{M}_S)$ is the smoothing parameter of the node $e$.
\end{enumerate}
\end{definition}

We will usually use subscript notation for lengths, i.e. $\delta_e$. See Figure \ref{fig:dual_graph_example} for an example of tropicalization. Note that the vertices incident to an edge need not be distinct: consider the nodal cubic.

\begin{figure} 
\centering
  \includegraphics{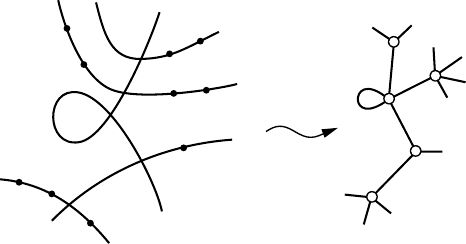}
  \caption{A typical log curve and its tropicalization.}
  \label{fig:dual_graph_example}
\end{figure}

\begin{definition} \label{def:msbar_order}
For any fs log scheme $S$, we give the monoid $\Gamma(S, \ol{M}_S)$ a partial order by taking $\Gamma(S, \ol{M}_S)$ as the positive cone,
i.e. $\alpha \geq \beta$ if and only if there is $\gamma \in \Gamma(S, \ol{M}_S)$ so that $\alpha = \beta + \gamma$.
\end{definition}

Let $\pi : C \to S$ be an arbitrary log curve and let $\ol{s} \in S$ be a geometric point. If $p$ and $q$ are nodes of $C_{\ol{s}}$ so that $\delta_p \geq \delta_q$, then the resulting relation $\delta_p = \gamma + \delta_q$ in $\ol{M}_S$ can be lifted \'etale locally to $M_S$
and pushed to $\mathscr{O}_S$. It follows that $p$ and $q$ have \'etale neighborhoods \'etale locally of the form $x_py_p = t_p$ and $x_qy_q = t_q$ where $t_p = st_q$ for some $s \in \mathscr{O}_S$. Then where $p$ smoothes, $q$
also smoothes. That is, $\delta_q \leq \delta_p$ means ``$q$ smoothes before $p$.''

\subsection{Line bundles from the log structure}

Given a scheme $X$ and an \'etale $X$-scheme $U$, denote by $\Div_X(U)$ the monoid of isomorphism classes of pairs $(\mathscr{L}, s)$ where
$\mathscr{L}$ is an invertible sheaf on $U$, $s \in \mathscr{L}(U)$, and the operation is given by tensor product.

If $X$ is an integral log scheme, then associated to each section of $\Gamma(U, \ol{M}_X)$ is an element of $\Div_X(U)$, whose construction we will
give in a moment. This construction defines a morphism of monoids from $\Gamma(U, \ol{M}_X)$ to $\Div_X(U)$.
Under sufficiently nice hypotheses, such as integrality, the specification of such a morphism for each $U$ is equivalent to the specification
of a log structure \cite{borne_vistoli}. An integral log scheme $X$ may therefore be thought of as a scheme with a distinguished family of
generalized Cartier divisors parametrized by $\ol{M}_X$, a point of view which will be important for our work.

\begin{definition} \hfill
\begin{enumerate}
  \item (Associated line bundles)
Let $(X, \epsilon : M_X \to X)$ be an integral log scheme, let $U \to X$ be an \'etale map, and let $\sigma \in \Gamma(U, \ol{M}_X)$.
Since the log structure on $X$ is integral, the action of $\mathscr{O}_X^*$ on $M_X$ is free, so the sheaf of lifts of $\sigma$ to $M_U$ is an
$\mathscr{O}_U^*$-torsor, $T(\sigma)$. We define $\mathscr{O}_U(\sigma)$, the \emph{invertible sheaf associated to $\sigma$}, to be the dual of the invertible
sheaf associated to $T(\sigma)$.

If $\sigma, \tau \in \Gamma(U, \ol{M}_X)$,
the \emph{invertible sheaf associated to $\sigma - \tau \in \Gamma(U, \ol{M}_X^{gp}$)} is given by
\[
  \mathscr{O}_U(\sigma - \tau) = \mathscr{O}_U(\sigma) \otimes \mathscr{O}_U(\tau)^\vee.
\]

  \item (Associated maps)
Now suppose $\tau \in \Gamma(U, \ol{M}_X^{gp})$ and $\sigma \in \Gamma(U, \ol{M}_X)$. Restricting the map $\epsilon : M_X|_U \to \mathscr{O}_U$ gives an $\mathscr{O}^*_U$-equivariant map $T(\sigma) \to \mathscr{O}_U$. Taking the associated map of invertible sheaves,
we have a cosection
\[
   \sigma : \mathscr{O}_U(\sigma)^\vee \to \mathscr{O}_U.
\]
Twisting by $\mathscr{O}_U(\sigma + \tau)$ gives an induced map of invertible sheaves
\[
  \sigma : \mathscr{O}_U(\tau) \to \mathscr{O}_U(\sigma + \tau), 
\]
which we denote by $\sigma$.
In the case that $\tau = 0$, the resulting morphism
\[
  \sigma : \mathscr{O}_U \to \mathscr{O}_U(\sigma)
\]
is called the \emph{canonical section of $\mathscr{O}_U(\sigma)$},
\end{enumerate}
\end{definition}

The induced morphism $\Gamma(U, \ol{M}_X) \to \Div_X(U)$ is functorial in $X$ in the sense that for any morphism $f : Y \to X$ of log schemes,
the diagram
\[
\xymatrix{
  \Gamma(U, \ol{M}_X) \ar[r] \ar[d]_{f^*} & \Div_X(U) \ar[d]^{f^*} \\
  \Gamma(f^{-1}(U), \ol{M}_Y) \ar[r] & \Div_Y(U) 
}
\]
commutes.

It can be helpful to introduce coordinates. Since Zariski and \'etale
$\mathscr{O}_U^*$-torsors coincide, there exists a Zariski cover $\{ U_i \to U \}_{i \in I}$ of $U$ and sections $\sigma_i \in \Gamma(U_i, M_X)$ so that
$\ol{\sigma_i} = \sigma|_{U_i}$. Since the log structure is integral, for each $i,j$ there is a unique $\xi_{i,j} \in \mathscr{O}_X^*(U_i \cap U_j)$ so
that $\sigma_i|_{U_i \cap U_j} = \xi_{i,j} \cdot \sigma_j|_{U_i \cap U_j}$. Since $\epsilon$ is the identity on $\mathscr{O}^*$, it follows that
$\epsilon(\sigma_i)$ and $\epsilon(\sigma_j)$ differ by $\xi_{i,j}$ on $U_i \cap U_j$. In the case that $\epsilon(\sigma_i)$ as a non--zero divisor
for each $i$, it follows that $(\epsilon(\sigma_i))_{i \in I}$ defines an effective Cartier divisor on $X$. A different choice of $\sigma_i$s gives a
rationally equivalent Cartier divisor. Whether $(\epsilon(\sigma_i))_{i \in I}$ defines a Cartier divisor or not, the tuple $(\xi_{i,j})$ automatically
satisfies the cocycle condition, and therefore it gives us gluing data for the line bundle $\mathscr{O}_U(\sigma)$.
Explicitly,
\[
  \mathscr{O}_U(\sigma)(V) = \left\{ (f_i)_{i \in I} \in \prod_{i \in I} \mathscr{O}_X(U_i \cap V)  \, \middle| \, f_i|_{U_i \cap U_j \cap V} = \xi_{i,j}|_V \cdot f_j|_{U_i \cap U_j \cap V} \text{ for all }i,j \in I \right\}.
\]
We will abbreviate this formula in the future to the more mnemonic, if notationally abusive,
\[
  \mathscr{O}_U(\sigma) = \left\{ (f_i) \, \middle| \, f_i = \epsilon\left(\frac{\sigma_i}{\sigma_j}\right) f_j \right\}.
\]
With respect to this presentation of $\mathscr{O}_U(\sigma)$, the canonical section is $(\epsilon(\sigma_i))$.
When $(\epsilon(\sigma_i))_{i \in I}$ represents a Cartier divisor $D$, the line bundle $\mathscr{O}_U(\sigma)$ is the same as $\mathscr{O}_U(D)$,
and the canonical sections of each agree.

Similarly, if $\sigma, \tau \in \Gamma(U, \ol{M}_X)$, the line bundle associated to $\sigma - \tau \in \Gamma(U, \ol{M}_X^{gp})$ is given by
\[
  \mathscr{O}_U(\sigma - \tau) = \left\{ (f_i) \, \middle| \, f_i = \epsilon\left(\frac{\sigma_i}{\sigma_j}\cdot \frac{\tau_j}{\tau_i}\right) f_j \right\}.
\]
If $\sigma \in \Gamma(U, \ol{M}_X^{gp})$ and $\tau \in \Gamma(U, \ol{M}_X)$, the induced morphism
$\mathscr{O}_U(\sigma) \overset{\tau}{\to} \mathscr{O}_U(\sigma + \tau)$, is given by $(f_i) \mapsto (\epsilon(\tau_i) f_i).$ 

\begin{definition}
Suppose $X$ is an integral log scheme. If $\sigma \in \Gamma(X, \ol{M}_X)$, the image of the map $\mathscr{O}_X(-\sigma) \overset{\sigma}{\to} \mathscr{O}_X$ is an ideal sheaf for a closed subscheme $|\sigma|$ of $X$, which we call the \emph{support} of $\sigma$.
Moreover, we have an exact sequence
\[
  \mathscr{O}_X(-\sigma) \overset{\sigma}{\to} \mathscr{O}_X \to \mathscr{O}_{|\sigma|} \to 0.
\]
\end{definition}

If $(\epsilon(\sigma_i))$ is a Cartier divisor, then $|\sigma|$ agrees with the support of the Cartier divisor. In this way, each section of the characteristic sheaf of $X$ determines a subscheme which behaves like an effective Cartier divisor. In the case of a log curve, these sections of
the characteristic sheaf and their associated line bundles have particularly nice interpretations in terms of piecewise linear functions, which we now define.

\begin{definition}
A \emph{piecewise linear function} on a tropical curve $\Gamma$ consists of the data
\begin{enumerate}
  \item For each $v \in V(\Gamma)$, an element $f_v \in \Gamma(S, \ol{M}_S)$;
  \item For each $h \in H(\Gamma)$, a natural number $n_h \in \mathbb{N}$,
\end{enumerate}
such that whenever $v,w$ are a pair of vertices incident to a common edge, $e$, $f_v - f_w$ is an integer multiple of $\delta_e$. The set of all piecewise linear functions on $\Gamma$ will be denoted $PL(\Gamma)$.
\end{definition}

The following description of the sections of the characteristic sheaf of a logarithmic curve slightly generalizes \cite[Remark 7.3]{ccuw}, though the idea seems to be implicit in the literature. We provide a somewhat heavy proof in the appendix in terms of ``uniform sets of charts."

\begin{proposition} \label{prop:smoothed_pw_linear} Let $\pi: C \to S$ be a log curve, and let $\ol{s}$ be a geometric point of $S$. Let $\Gamma$ be the tropicalization of the fiber
$C_{\ol{s}}$ of $C$ over $\ol{s}$. Then there is an \'etale neighborhood $U$ of $\ol{s}$ and a bijection
\begin{align*}
  \Gamma(\pi^{-1}(U), \ol{M}_C) &\overset{\sim}{\longrightarrow} PL(\Gamma) \\
   \sigma & \mapsto PL(\sigma).
\end{align*}

If $\sigma \in \Gamma(\pi^{-1}(U), \ol{M}_C)$, the corresponding element $PL(\sigma) = ((f_v), (n_h))$ of $PL(\Gamma)$ is determined by taking $f_v$ to be the stalk of $\sigma$ at the generic point of $v$ for all $v \in V(\Gamma)$, and taking $n_h$ to be the slope of
$\sigma$ at $h$ for all $h \in H(\Gamma)$.

If $\pi : C \to S$ is a log curve over the spectrum of an algebraically closed field and $\sigma \in PL(\Gamma)$, then $|\sigma|^{red}$ is the union of the marked points at which $\sigma$ has non-zero slope and the components of $C$
where $\sigma$ has non-zero value.
\end{proposition}

Finally, we recall a description of associated line bundles for log curves over geometric points.

\begin{proposition} \label{prop:log_curve_line_bundles}
(\cite[Proposition 2.4.1]{rsw})
Let $\pi : C \to S$ be a log curve over $S$, where the underlying scheme of $S$ is the spectrum of an algebraically closed field. Let $\sigma$ be a global section of $\ol{M}_C$ with corresponding piecewise linear function $( (f_v), (n_h) )$.
Let $v$ be a component of $C$. Then
\[
  \mathscr{O}_C(\sigma)|_{v} = \mathscr{O}_{v}\left( \sum_p \mu_p p \right) \otimes \pi^*\mathscr{O}_S(f_v)
\]
where the sum is over the edges and half-edges $p$ incident to $v$, and $\mu_p$ is the ``outgoing slope" of the piecewise linear function at the point $p$: when $p$ is a marked point, this is the integer $n_p$; when $p$ is a node joining $v$ to $w$,
this is $(f_w - f_v)/\delta_p$.
\end{proposition}

For example, in Figure \ref{fig:mesa_example}, $\mathscr{O}_C(\ol{\lambda})$ restricts on each component from left to right to $\mathscr{O}(p_1)$, $\mathscr{O}(-p_1)$, $\mathscr{O}(-p_3)$, $\mathscr{O}(p_3 - p_4 - p_5)$, $\mathscr{O}(p_4)$, and $\mathscr{O}(p_5)$.

\begin{figure}
\begin{center}
\includegraphics[width=\textwidth]{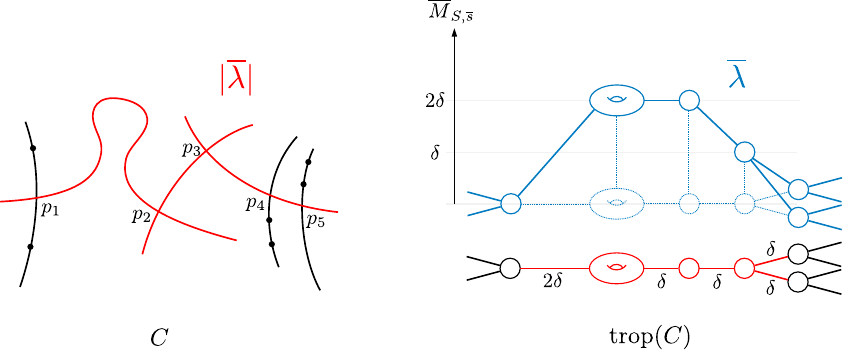}
\end{center}
\caption{A log curve $C$, a section $\ol{\lambda}$ of its characteristic sheaf, and the support $|\ol{\lambda}|$ of that section.}
\label{fig:mesa_example}
\end{figure}

\section{Mesa curves}
\label{sec:mesa_curves}

Let $C$ be a log smooth curve over the spectrum of an algebraically closed field. By a subcurve of $C$, we mean a union of some of its irreducible components. If $E$ is a connected subcurve of $C$ with positive genus, the \emph{core of $E$} is the minimal connected subcurve $F$ of $E$ with the same arithmetic genus as $E$.

\begin{definition}
Let $\Gamma$ be a tropical curve. A \emph{path} $W$ in $\Gamma$ is a sequence $v_0e_1v_1e_2 \cdots e_kv_k$ of vertices and edges in $\Gamma$ so that the vertices $v_i$ are distinct and $v_{i - 1}$ and $v_i$ are the ends of the edge $e_i$ for all $i$. Given subsets $A$ and $B$
of $V(\Gamma)$, we say that $W$ is a path from $A$ to $B$ if $v_0 \in A$, $v_k \in B$, and $v_i \not\in A \cup B$ for $i \neq 0, k$.
\end{definition}

\begin{definition} \label{def:simple_mesa}
A section $\ol{\lambda} \in \Gamma(C, \ol{M_C})$ is a \emph{simple mesa} if
\begin{enumerate}
  \item $PL(\ol{\lambda})$ has slope 0 on the half-edges of $\Gamma(C)$
  \item the support $E$ of $\ol{\lambda}$ is either empty or connected with positive genus;
  \item $PL(\ol{\lambda})$ is constant on $\core(E)$;
  \item each component of $E$ lies on a path from $\core(E)$ to the complement of $E$;
  \item the restriction of $PL(\ol{\lambda})$ to any path from $\core(E)$ to the complement of $E$ is weakly decreasing
          with slopes $0$ or $-1$;
  \item for each rational component $Z$ of $E$, $\mathscr{O}_Z(-\ol{\lambda})$ has non-negative degree;
  \item $H^1(E, \mathscr{O}_E(-\ol{\lambda})) = 0$.
\end{enumerate}
\end{definition}
See Figure \ref{fig:mesa_example} for an example of a simple mesa. Note that with respect to the partial order of Definition \ref{def:msbar_order} on $\ol{M}_S$, $PL(\ol{\lambda})$ has a maximum value and it is acheived on the core of $E$. We call this maximum the \emph{elevation} of $\ol{\lambda}$, and we call the union of the components of $E$
where $PL(\ol{\lambda})$ acheives this maximum the \emph{top} of $\ol{\lambda}.$ Note that the top of $\ol{\lambda}$ contains the core of $E$,
possibly properly. A path from $\core(E)$ to the complement of $E$ is \emph{steep} if its slopes are all $-1$. A mesa is \emph{steep}
if all such paths are steep. We call the $H^1(E, \mathscr{O}_E(-\ol{\lambda})) = 0$ condition \emph{acyclicity}.

A mesa curve extends these notions to families.

\begin{definition} \label{def:mesa_curve}
A \emph{mesa curve} over a scheme $S$ consists of
\begin{enumerate}
  \item a proper log curve $\pi : C \to S$,
  \item a section $\ol{\lambda} \in \Gamma(C, \ol{M}_C)$,
\end{enumerate}
such that, for each geometric point $x$ of $S$, the pullback of $\ol{\lambda}$ to $C \times_S x$ is of the form $\sum_{i = 1}^{k} \ol{\lambda}_{i}$ for some $k \geq 0$, where the $\ol{\lambda}_{i}$ are simple mesas with disjoint support.

Given a mesa curve $(\pi : C \to S, \ol{\lambda})$ we say $\ol{\lambda}$ is a mesa and $E = |\ol{\lambda}|$, extending our conventions over geometric points. If the pullback of $E$ to each geometric fiber of $C$ has at most one connected component, we say that $C$ is a \emph{simple} mesa curve.
\end{definition}

Mesa curves are a generalization of the centrally aligned curves of \cite{rsw}. For comparison, we recall the definition of centrally aligned curve.
Let $\pi : C \to S$ be a log curve of genus one.
There is a global section $\lambda \in \Gamma(C, \ol{M}_C)$ that measures distance from the core of
$C$ in each fiber \cite[Section 3.3]{rsw}. Then a centrally aligned curve consists of (c.f. \cite[Definition 4.6.2.1]{rsw})
\begin{enumerate}
  \item a proper genus 1 log curve $\pi : C \to S$ and
  \item a section $\rho \in \Gamma(S, \ol{M}_S)$
\end{enumerate}
such that for each geometric point $\ol{s}$ of $S$,
\begin{enumerate}
  \item $\rho_{\ol{s}}$ is comparable to $\lambda(v)$ for each vertex $v \in \Gamma(C_{\ol{s}})$,
  \item the subcurve $E_{\ol{s}}$ of $C_{\ol{s}}$ where $\lambda < \rho$ is stable, and
  \item given any pair of vertices $v, w \in V(\trop(E_{\ol{s}}))$, the distances $\lambda(v)$ and $\lambda(w)$ are comparable in the partial
  order on $\ol{M}_S$.
\end{enumerate}

We may interpret a centrally aligned curve $(C, \rho)$ as a simple mesa curve by taking $\ol{\lambda} = \max \{ \rho - \lambda, 0 \}$.
The resulting mesa is necessarily steep, and Lemma \ref{lem:acyclicity_criterion} assures us that it is acyclic. That is,
to move from centrally aligned curves to mesa curves
we trade the datum $\rho \in \Gamma(S, \ol{M}_S)$ for the more flexible datum $\ol{\lambda} \in \Gamma(C, \ol{M}_C)$.

\medskip

When $C$ is a simple mesa curve, there is only one elevation $\rho$ to worry about in each fiber. In fact, these $\rho$s glue together to a global section of $\ol{M}_S$.

\begin{lemma} \label{lem:rho_is_global}
Let $\pi : C \to S$ be a simple mesa curve, and for each geometric point $\ol{s}$ of $S$, let $\rho_{\ol{s}}$ be the elevation of the restriction of 
$\ol{\lambda}$ to the fiber over $\ol{s}$.
Then there is a global section $\rho \in \ol{M}_S(S)$ restricting to $\rho_{\ol{s}}$ for each geometric point $\ol{s}$ of $S$.
\end{lemma}
\begin{proof}
For each geometric point $\ol{s}$, construct an \'etale neighborhood $U_{\ol{s}}$ of $\ol{s}$ admitting a uniform set of charts (Definition \ref{def:unif_charts}). Then following the isomorphisms
$\ol{M}_{S, \ol{s}} \cong P \cong \Gamma(U_{\ol{s}}, \ol{M}_S)$ gives a unique extension of $\rho_{\ol{s}} \in \ol{M}_{S,\ol{s}}$ to a section $\rho^{\ol{s}} \in \Gamma(U_{\ol{s}}, \ol{M}_S)$. Following the values of $\rho$ on the charts $V_i$, at each geometric point $\ol{t}$ in $U^{\ol{s}}$, the stalk of $\rho^{\ol{s}}$ at $\ol{t}$ agrees with $\rho_{\ol{t}}$. Then the sections $\rho^{\ol{s}}$ glue to the required global section.
\end{proof}

If $\pi : C \to S$ is a simple mesa curve and $\ol{s}$ is a geometric point of $S$, we note that $\rho_{\ol{s}} = 0$ if and only if the fiber of $E$ over $\ol{s}$ is empty. Note that then $\rho : \mathscr{O}_{\ol{s}}(-\rho) \to \mathscr{O}_{\ol{s}}$
is an isomorphism if and only if $\rho_{\ol{s}} = 0$, and $\mathscr{O}_{\ol{s}}(-\rho) \to \mathscr{O}_{\ol{s}}$ is the zero map otherwise. In other words, the support of $\rho$ is the locus in $S$ over which $E$ is non-empty.

\section{Construction of the contraction map} \label{sec:contraction_construction}

\subsection{Intuition and strategy}

Start by assuming that $C$ is simple. A tempting way forward is to construct $\ol{C}$ as a topological space by collapsing $E$ to a point fiberwise, then to take $\mathscr{O}_{\ol{C}} = \tau_*\mathscr{O}_C$. In general, this fails to produce a curve with singularities of the correct genus.

\begin{example}
Consider a stable curve $C \to \Spec k$ consisting of a smooth elliptic component, $E$, with two attached rational components. Give the base the log structure associated to the chart $\N\delta \to k$ sending $\delta \mapsto 0$, and give $C$ a log
structure with both smoothing parameters equal to $\delta$. Setting $\ol{\lambda}$ equal to $\delta$ on $E$ and 0 on the rational
components defines a mesa. If we let $\tau : C \to \ol{C}$ be the quotient map of topological spaces collapsing $E$ to a point,
then $(\ol{C}, \tau_*\mathscr{O}_C)$ is a genus 0 curve with two components meeting at an ordinary double point.
\end{example}

\begin{example}
Let $S = \Spec k[t]/(t^2)$ and give $S$ the log structure associated to the chart $\N\delta \to k[t]/(t^2)$ sending $\delta \mapsto t$.
Let $C$ be a mesa curve over $S$ restricting to the curve of the previous example, in which the nodes are beginning to smooth out. Let $\mathscr{O}_{C_0} = \mathscr{O}_C \otimes k$.

Consider the sequence
\[
  0 \to \mathscr{O}_{C_0} \overset{t}{\to} \mathscr{O}_C \to \mathscr{O}_{C_0} \to 0,
\]
and its pushforward,
\[
  0 \to \tau_*\mathscr{O}_{C_0} \to \tau_*\mathscr{O}_C \to \tau_*\mathscr{O}_{C_0}.
\]
where $\tau : C \to \ol{C}$ is the topological quotient.

We would like to compute $(\tau_*\mathscr{O}_C) \otimes k$. This is the same as the image of $\tau_*\mathscr{O}_C$ in
$\tau_*\mathscr{O}_{C_0}$.
Let $E$ be the elliptic component and let $Z_1$, $Z_2$ be the rational components. Let $p_1, p_2$ be the respective points where $E$ meets
$Z_1$ and $Z_2$. Let $U = C - (Z_1 \cup Z_2)$. Since $\tau$ is an isomorphism in the complement of $E$, we may restrict our attention to
the geometric stalk of the sequence $0 \to \tau_*\mathscr{O}_{C_0} \to \tau_*\mathscr{O}_C \to \tau_*\mathscr{O}_{C_0}$ at $\ol{\tau(E)}$.
Let $f \in (\tau_*\mathscr{O}_{C_0})_{\ol{\tau(E)}}$. Since $U$ is affine, finding lifts of $f$ to $\tau_*\mathscr{O}_C$ is equivalent to
finding compatible lifts of the restriction $f|_U$ to $\mathscr{O}_C(U)$ and of the stalks $f_{\ol{p_i}}$ to $\mathscr{O}_{C, \ol{p_i}}$.

There are units $\alpha_i$ so that the \'etale stalk $(\mathscr{O}_C)_{\ol{p_i}}$ is isomorphic to $(\mathscr{O}_{S}[x_i, y_i]/(x_iy_i - \alpha_it))_{\ol{(x_i, y_i)}}$, where $x_i$
vanishes on $Z_i$ and $y_i$ vanishes on $E$.
A lift of $f_{\ol{p_i}}$ is of the form
\[
  f_{\ol{p_i}}(x_i, y_i) + tg_{\ol{p_i}}(x_i, y_i).
\]
On $U$, $\mathscr{O}_C(U)  = \mathscr{O}_{C_0}(U)[t]/(t^2)$,
so a lift of $f|_U$ is of the form
\[
  f|_U + tg_U.
\]
Write $f_{\ol{p_i}} = a^i_{0,0} + a^i_{1,0}x_i + a^i_{2, 0}x_i^2 + \cdots + a^i_{0,1}y_i + a^i_{0,2}y_i^2 + \cdots$.  Consider the restriction
of $f_{\ol{p_i}}(x_i, y_i) + tg_{\ol{p_i}}(x_i, y_i)$ to the complement of $Z_i$. Then $x_i$ is invertible in the complement of $Z_i$, so we have
\[
  y_i = \alpha_i t/x_i
\]
on this locus. The restriction of $f_{\ol{p_i}}(x_i, y_i) + tg_{\ol{p_i}}(x_i, y_i)$ to the complement of $Z_i$ may then be written
\[
  a^i_{0,0} + a^i_{1,0}x_i + \cdots + t(\alpha_ia^i_{0,1}x_i^{-1} + g_{\ol{p_i}}(x, 0))
\]
Therefore the problem of finding compatible lifts of the $f$s is the same as that of finding a rational function $g$ on $E \otimes k$ with
pole $\alpha_1 a^1_{1,0}$ at $p_1$ and pole $\alpha_2 a^2_{1,0}$ at $p_2$. This is the Mittag-Leffler problem; it is classical that a solution exists if and only if
$c_1\alpha_1a^1_{1,0} + c_2\alpha_2a^2_{1,0} = 0$, for some nonzero constants $c_1, c_2$. That is, $(\ol{C}, (\tau_*\mathscr{O}_C) \otimes k)$ has a tacnode at
$\tau(E)$. This is a genus 1 singularity, as desired.
\end{example}

These examples suggest that the issue is the failure of $\tau_*$ to commute with base change in $S$ when applied to $\mathscr{O}_C$, so our strategy will be to find
constituent parts of the desired structure sheaf whose formation commutes with base change in $S$. Let $\pi : C \to S = \Spec A$ be a simple mesa curve, let $E$ be the support of $\ol{\lambda}$, and let $\rho$ be the global section of $\ol{M}_S$ in Lemma \ref{lem:rho_is_global}.

When $\mathscr{O}_C(-\ol{\lambda}) \to \mathscr{O}_C$ is injective, $\mathscr{O}_C(-\ol{\lambda})$ is the ideal sheaf of $E$ in $C$. It is reasonable
to expect that $\tau_*\mathscr{O}_C(-\ol{\lambda})$ will be the ideal sheaf of the point $\tau(E)$ in $\ol{C}$. If $U$ is a neighborhood of
$E$ small enough that $\tau(U)$ is an affine neighborhood of $\tau(E)$, this ideal sheaf will have values
$\Gamma(U, \mathscr{O}_C(-\ol{\lambda}))$. We will be able to conclude from acylicity of $\ol{\lambda}$ that the formation of this module
commutes with base change in $S$.
This will be the first constituent of $\mathscr{O}_{\ol{C}}$. 

Since $\Gamma(U, \mathscr{O}_C(-\ol{\lambda}))$ is supposed to be the ideal sheaf of a point, the rest of $\mathscr{O}_{\ol{C}}$ should be generated by the constant functions, i.e. $\ol{\pi}^{-1}\mathscr{O}_S$.
On a sufficiently small neighborhood of $\tau(E)$, the sections of this sheaf should be identifiable with $\Gamma(S, \mathscr{O}_S)$. This will be the second constituent of $\mathscr{O}_{\ol{C}}$.

At $\tau(E)$, we expect the constant functions and the ideal sheaf of $\tau(E)$ to intersect in the constant functions that are zero on $E$, i.e.,
$\Gamma(S, \mathscr{O}_S(-\rho))$.

\medskip

We now turn this intuition into the construction of a suitable ring. Let $U$ be any open subset of $C$ containing $E$. (We will be more particular about $U$'s identity later, but we leave it variable here so that we will have the language to compare different choices of $U$.)

There is a non-unital $\mathscr{O}_C$-bilinear multiplication on $\mathscr{O}_C(-\ol{\lambda})$ induced by the log structure on $C$:
\[
  \mathscr{O}_C(-\ol{\lambda}) \otimes_{\mathscr{O}_C} \mathscr{O}_C(-\ol{\lambda}) \cong \mathscr{O}_C(-2\ol{\lambda}) \overset{\ol{\lambda}}{\to} \mathscr{O}_C(-\ol{\lambda}).
\]

It may be helpful to be more explicit: recall that the sections of $\mathscr{O}_C(-\ol{\lambda})$ on $U$ can be presented as
\[
  \Gamma(U, \mathscr{O}_C(-\ol{\lambda})) = \left\{ (f_i) \,\middle|\, f_i = \epsilon\left( \frac{\ol{\lambda}_j}{\ol{\lambda}_i}\right) f_j \right\}.
\]
In these coordinates, the multiplication is $(f_i) \cdot (g_i) = (\epsilon(\ol{\lambda}_i)f_ig_i).$
Given sections $f, g \in \mathscr{O}_C(-\ol{\lambda})$, we will denote this product by $\ol{\lambda}(fg)$. Note that $\ol{\lambda}(fg) = \ol{\lambda}(f)g = f\ol{\lambda}(g)$, rather than $\ol{\lambda}(f)\ol{\lambda}(g)$.

Let $B(U) = \Gamma(U, \mathscr{O}_C(-\ol{\lambda})) \oplus \Gamma(S, \mathscr{O}_S)$. Since $\mathscr{O}_C(-\ol{\lambda})$ is a $\pi^{-1}\mathscr{O}_S$ module, we can
give $B(U)$ a ring structure by the rule
\[
  (f,c) \cdot (g, d) = (\ol{\lambda}(fg) + df + cg, cd).
\]
Now give $\Gamma(S, \mathscr{O}_S(-\rho))$ a $B(U)$-module structure by $(f, c) \cdot g = cg$. Consider the map
\begin{align*}
  \Gamma(S, \mathscr{O}_S(-\rho)) \, &\longrightarrow \, B(U) \\
  g \quad &\longmapsto \, ((\rho - \ol{\lambda})(g), -\rho(g)).
\end{align*}
Note that the minus sign in $(\rho - \ol{\lambda})(g)$ is in $\ol{M}_C$ while the minus sign in $-\rho(g)$ in $B(U)$.
This map is a $B(U)$-module homomorphism:
\begin{align*}
  (f, c) \cdot ((\rho - \ol{\lambda})(g), -\rho(g)) &= (\ol{\lambda}(f(\rho - \ol{\lambda})(g)) - f\rho(g) + c (\rho - \ol{\lambda})(g), -c\rho(g)) \\
   &= (f(\ol{\lambda} + \rho - \ol{\lambda})(g) -f\rho(g) + (\rho - \ol{\lambda})(cg), -\rho(cg)) \\
   &=( (\rho - \ol{\lambda})(cg), -\rho(cg) ).
\end{align*}
In particular, its image is an ideal of $B(U)$. Let $\ol{B}(U)$ be the quotient ring.

We now have an exact sequence of $A$-modules
\[
  \Gamma(S, \mathscr{O}_S(-\rho)) \to B(U) \to \ol{B}(U) \to 0
\]
in which the last two modules are $A$-algebras. We will call this the \emph{$\ol{B}$-sequence}. (In Proposition \ref{prop:bbar_sequence_is_nice} we will see that this sequence is short exact for a good choice of $U$.)

Let $\ol{U} := \Spec \ol{B}(U)$.
There is a map of $A$-algebras $B(U) \to \Gamma(U, \mathscr{O}_C)$ induced by $\ol{\lambda}$ on the first summand and $\pi^\sharp$ on the second.
The composite
\[
  \Gamma(S, \mathscr{O}_S(-\rho)) \to B(U) \to \Gamma(U, \mathscr{O}_C)
\]
is zero, so the morphism $B(U) \to \Gamma(U, \mathscr{O}_C)$ descends to an $A$-algebra morphism $\ol{B}(U) \to \Gamma(U, \mathscr{O}_C)$. Taking $\Spec$, we obtain a commutative triangle
\begin{equation} \label{eq:contraction_triangle_for_U}
  \xymatrix{
    U \ar[r]^\tau \ar[rd]_\pi & \ol{U} \ar[d]^{\ol{\pi}} \\
    & S.
  }
\end{equation}

We now state the construction of $\ol{C}$ in what we will call the ``standard situation." We will be able to show by the end of the next subsection that this
determines $\ol{C}$ in general.

\begin{situation}
  \begin{enumerate}
   \item $S = \Spec A$ is a Noetherian, finite dimensional, affine scheme.
   \item $(\pi : C \to S , \ol{\lambda} )$ is a simple mesa curve so that $C$ is a projective scheme.
   \item $E$ is the support of $\ol{\lambda}$.
   \item $\rho$ is the global section of $\ol{M}_S$ from Lemma \ref{lem:rho_is_global}.
   \item $\sigma_1, \ldots, \sigma_h$ are sections of $\pi$ so that each component of each geometric fiber of $C - E \to S$ meets at least one $\sigma_i$.
   \item $U$ is the complement of the $\sigma_i$s in $C$.
  \end{enumerate}
\end{situation}

\begin{definition} \label{def:contraction}
Assume that we are in the standard situation. Let $\ol{U} = \Spec \ol{B}(U)$, as above. The restriction of $\tau$ to $U - E$ will turn out to be an open immersion
(Proposition \ref{prop:iso_outside_E}). We define the contracted curve $\ol{C}$ as the pushout
\[
  \xymatrix{
    U - E \ar[r] \ar[d] & C - E \ar[d] \\
    \ol{U} \ar[r] & \ol{C}.
  }
\]

When we wish to emphasize the sections $\sigma_i$ in use, we will write $\ol{(C, \sigma_\bullet)}$ instead of $\ol{C}$.
\end{definition}

Note that the commutative triangle (\ref{eq:contraction_triangle_for_U}) induces a commutative triangle
\[
  \xymatrix{
    C \ar[r]^\tau \ar[rd]_\pi & \ol{C} \ar[d]^{\ol{\pi}} \\
    & S.
  }
\]

\subsection{Reduction to the standard situation}

We first reduce to the case that $C \to S$ is simple. Suppose that Theorem \ref{thm:contraction} holds for simple mesa curves, and let $\pi : C \to S$ be an arbitrary mesa curve.
We begin by constructing $C \to \ol{C}$ in an \'etale neighborhood of each geometric point $\ol{s}$ of $S$.
Suppose that $S$ is an \'etale neighborhood of $\ol{s}$ admitting a uniform set of charts (Definition \ref{def:unif_charts}).
Then the decomposition $\ol{\lambda} = \sum_{i = 1}^k \ol{\lambda}_{E_i}$ holds globally. Construct contractions $\tau_i : C \to \ol{C}_i$ of the simple mesa curves
$(C, \ol{\lambda}_{E_i})$. Note that $C - {E_i} \to \ol{C}_i$ is an open immersion, so we can identify $E_j$ with its image in $\ol{C}_i$, $j \neq i$. Then we can define
\[
  \ol{C} = \colim \left\{
    \vcenter{\vbox{\xymatrix{
      & C - \bigcup_{i} E_i \ar[dl] \ar[d] \ar[drr] \\
      \ol{C}_1 - \bigcup_{i \neq 1} E_i & \ol{C}_2 - \bigcup_{i \neq 2} E_i & \cdots & \ol{C}_k - \bigcup_{i \neq k} E_i
    }}}
  \right\}
\]
where each arrow is the open immersion induced by $\tau_i$. This satisfies the conclusions of Theorem \ref{thm:contraction}. Since the contraction construction commutes with base change for simple mesa curves, \'etale descent gives us
the contraction for general mesa curves.

We next claim that by working locally enough in $S$, we may work in the standard situation.

\begin{lemma} \label{lem:etale_local_reductions}
Let $\pi : C \to S$ be a simple mesa curve, and let $\ol{s}$ be a geometric point of $S$.

Then there is an \'etale neighborhood $S'$ of $\ol{s}$ and geometric point $\ol{s}'$ over $\ol{s}$ so that
\begin{enumerate}
  \item $S' = \Spec A$ is affine;
  \item $C' = C \times_S S'$ is a projective scheme;
  \item There are sections $\sigma_1, \ldots, \sigma_h : S' \to C'$ of $\pi|_{S'}$ meeting each component of
$(C' - E \times_S S') \cap \pi^{-1}(\ol{s}')$;
  \item For all geometric points $\ol{t} \in S'$ and components $Z$ of $(C' - E \times_S S') \cap \pi|_{S'}^{-1}(\ol{t})$, one of the sections $\sigma_i$ meets
$Z$.
\end{enumerate}
\end{lemma}
\begin{proof}
Note that properties (ii), (iii), (iv), are stable under \'etale localization of $S$. It is well-known that there is an \'etale neighborhood $U_1$ of $\ol{s}$ so that (ii) and (iii) hold.

Construct an \'etale neighborhood $U$ of $\ol{s}$ admitting a uniform set of charts $\{ V_i \to C|_U \}_{i \in I}$ (Definition \ref{def:unif_charts}).
Suppose $Z$ is a component of $(C - E) \cap \pi^{-1}(\ol{t})$, where $\ol{t}$ is a geometric point of $U$. Let $\ol{u}$ be any smooth geometric point of $Z$. The description of $\ol{M}_C$ in the proof of Lemma \ref{lem:charts_for_dual_graph_DVR}
implies that there is a component $v$ of $C_{\ol{s}}$ so that $(\ol{\lambda}_v)_{\ol{t}} = \ol{\lambda}_{\ol{u}}$. On the other hand, since $Z$ is not a component of $E \cap \pi^{-1}(\ol{t})$, the stalk of $\ol{\lambda}$ at $\ol{u}$ is zero. By definition of
$\ol{\lambda}$, $\ol{\lambda}_v$ is a sum of smoothing parameters on a path from $v$ to a component $w$ of $C_{\ol{s}}$ outside of $E$. Since $(\ol{\lambda}_v)_{\ol{t}} = 0$, all of the nodes on this path smooth out over $\ol{t}$. This implies that
$Z$ meets the section $\sigma_j$ going through $w$, so we have (iv).

Finally, take an affine neighborhood $U_2$ of $\ol{s}$ inside $U \times_S U_1$. This gives us the desired \'etale neighborhood.
\end{proof}

\begin{lemma}
It suffices to construct $C \to \ol{C} \to S$ when $S$ is Noetherian and finite dimensional.
\end{lemma}
\begin{proof}
Every affine $S$ is a projective limit of Noetherian and finite dimensional rings. All of the properties we ask of $C \to \ol{C} \to S$ are preserved under projective limits.
\end{proof}

Next, we need to be able to glue our local contracted curves together. We state a couple of results early. Proposition \ref{prop:C_bar_commutes_base_change} will follow immediately from Corollary
\ref{cor:U_bar_commutes_base_change} and Proposition \ref{prop:iso_outside_E}. Proposition \ref{prop:section_invariant_early} is Proposition \ref{prop:section_invariant}.

\begin{proposition} \label{prop:C_bar_commutes_base_change}
Assume we are in the standard situation and let $f : T \to S$ be a morphism of affine, Noetherian, finite-dimensional schemes. Pull back the data of the standard situation along
$f$ to obtain $\pi_T : C_T \to T$, $E_T$, $\rho_T$, $\sigma_{T,1}, \ldots, \sigma_{T, h}$, and $U_T$ over $T$. These data are again in the standard situation, so we may apply
Definition \ref{def:contraction} to obtain a triangle
\[
  \xymatrix{
    C_T \ar[r] \ar[rd] & \ol{C_T} \ar[d] \\
    & T.
  }
\]

Then there is a uniquely specified isomorphism $\phi_{S,T} : \ol{C_T} \to \ol{C} \times_S T$ so that
\[
  \xymatrix{
    C_T \ar[rdd] \ar[rr] \ar[rd]& & \ol{C} \times_S T  \ar[ldd] \\
     & \ol{C_T} \ar[d] \ar^{\phi_{S,T}}[ru]\\
    & T
  }
\]
commutes. Moreover, these isomorphisms satisfy the cocycle condition in the sense that if $g : W \to T$ is a second
morphism of affine, Noetherian, finite-dimensional schemes,
\[
  \phi_{S,W} = \phi_{T,W} \circ \phi_{S,T}|_{W}.
\]
\end{proposition}

\begin{proposition} \label{prop:section_invariant_early}
Suppose we are in the standard situation and $\sigma'_1, \ldots, \sigma'_k$ are a second choice of sections. There is a uniquely
specified isomorphism $\phi_{\sigma_\bullet, \sigma_\bullet'} : \ol{(C, \sigma_\bullet)} \to \ol{(C, \sigma_\bullet')}$ so that
\[
  \xymatrix{
    C_T \ar[rdd] \ar[rr] \ar[rd]& & \ol{(C, \sigma_\bullet')}  \ar[ldd] \\
     & \ol{(C, \sigma_\bullet)} \ar[d] \ar^{\phi_{\sigma_\bullet, \sigma_\bullet'}}[ru]\\
    & T
  }
\]
commutes. Moreover, if $\sigma''_1, \ldots, \sigma''_l$ is a third choice of sections, the cocycle condition
\[
  \phi_{\sigma_\bullet, \sigma_\bullet''} = \phi_{\sigma_\bullet'', \sigma_\bullet'} \circ \phi_{\sigma_\bullet, \sigma_\bullet'}
\]
holds.
\end{proposition}

Now, suppose that $\pi : C \to S$ is a simple mesa curve where $S = \Spec A$ is Noetherian and finite-dimensional. Then there is an \'etale cover $S' \to S$ and sections of $\pi|_{S'}$ as in Lemma \ref{lem:etale_local_reductions}. The two preceding results
imply that the map $C' \to \ol{C'}$ over $S'$ descends to a map over $S$. Taking projective limits, we have contractions of mesa curves over all affine schemes in a manner that commutes with base change, so \'etale descent gives contractions of mesa curves over a general base.

\bigskip

This leaves a great deal to be verified. We must check that in the standard situation,
\begin{enumerate}
  \item The construction commutes with pullback in $S$.
  \item $\ol{C}$ is flat over $S$.
  \item $\tau$ is surjective.
  \item $\tau : U - E \to \ol{U} - \tau(E)$ is an isomorphism.
  \item $\ol{C}$ does not depend on the choice of sections $\sigma_i : S \to C$.
  \item If $S$ is a log point, $\ol{C}$ has a genus $g$ singularity at $\tau(E)$, which is elliptic Gorenstein if $\ol{\lambda}$ is a steep mesa with genus one support.
  \item $\ol{\pi}: \ol{C} \to S$ has reduced geometric fibers.
  \item $\ol{\pi} : \ol{C} \to S$ is proper.
\end{enumerate}

We spend the remainder of the paper establishing these facts in roughly this order.

\subsection{Values of $\mathscr{O}_E(-\ol{\lambda})$}

In this section, we will examine the values
that sections of $\mathscr{O}_E(-\ol{\lambda})$ on $E$ can take on the attachment points of $E$
to the rest of $C$. These turn out to be constrained by a Mittag-Leffler problem in a way that will later imply that $\tau(E)$ is a genus $g$ singularity.

For each reduced curve $D$ over an algebraically closed field $k$, sheaf $\mathscr{L} \in Pic(D)$, and point $p \in D$, fix an isomorphism $\mathrm{ev}_p : \mathscr{L} \otimes k(p) \to k$. Given a section
$\sigma \in \mathscr{L}(V)$ where $p \in V$, denote by $\sigma(p)$ the image of $\sigma$ under the composition $\mathscr{L}(V) \longrightarrow \mathscr{L} \otimes k(p) \overset{\mathrm{ev}_p}{\longrightarrow} k$.

We now state the main result of this section.

\begin{proposition} \label{prop:generalcurvesectionvalues}
Suppose that we are in the standard situation with $S$ the spectrum of an algebraically closed field $k$. Let $x_1, x_2, \ldots, x_m$ be the points where $E$ meets $\ol{C - E}$.
Consider the following problem:

\begin{center}
\fbox{
\parbox{.8\textwidth}{
Given $a_1, \ldots, a_m \in k$, find a section $\sigma \in \Gamma(E, \mathscr{O}_E(-\ol{\lambda}))$ so that
\[
   \sigma(x_i) = a_i, \quad 1 \leq i \leq m.
\]
}
}
\end{center}

There is a codimension $g$ condition on the $a_i$s under which a solution exists. Moreover, if there is a solution, then there is a 1-dimensional family of solutions.

If $E$ has genus one, then there are constants $c_1, \ldots, c_m \in k$ so that a solution exists if and only if
\[
   c_1a_1 + \cdots + c_ma_m = 0,
\]
and $c_i \neq 0$ if and only if $x_i$ is on a steep path from $\core(E)$.
\end{proposition}

This follows from the next lemma by taking $\mathcal{L} = \mathscr{O}_E(-\ol{\lambda})$, $E_1$ equal to the top of $\ol{\lambda}$, and
$E_2, \ldots, E_r$ to be the remaining connected subcurves of $E$ on which $\ol{\lambda}$ is constant. Proposition \ref{prop:log_curve_line_bundles} ensures that $\mathcal{L}$ restricts to the
$E_i$s as desired.
We remark that the dual graph of $E$ in the following statement becomes a tree after contracting $E_1$, since $E_1$ contains the core of $E$. This makes the points $q^{(i)}$ well-defined.

\begin{lemma} \label{lem:sectionvaluesinduction}
Let $E$ be a connected reduced nodal curve of positive genus over an algebraically closed field $k$. Let $E_1, \ldots, E_r$ be subcurves of $E$ so that
\begin{enumerate}
  \item $E$ is the union of $E_1, \ldots, E_r$;
  \item $E_1$ contains the core of $E$;
  \item $E_i$ and $E_j$ have no irreducible components in common when $i \neq j$.
\end{enumerate}
Let $\mathcal{L}$ be a line bundle on $E$ so that
\begin{enumerate}
  \item $H^1(E, \mathcal{L}) = 0$;
  \item $\mathcal{L}$ has non-negative degree on the irreducible components of $E$;
  \item $\mathcal{L} \otimes \mathscr{O}_{E_1} \cong \mathscr{O}_{E_1}(p_1^{(1)} + \cdots + p_{n_1}^{(1)})$ where $p_1^{(1)}, \ldots, p_{n_1}^{(1)}$ are distinct points that include the attachment points of $E_1$ to the rest of $E$;
  \item for $i \neq 1$, we have $\mathcal{L} \otimes \mathscr{O}_{E_i} \cong \mathscr{O}_{E_i}(p_1^{(i)} + \cdots + p_{n_i}^{(i)} - q^{(i)})$ where
   $q^{(i)}$ is the attachment point of $E_i$ ``heading towards" $E_1$ and $p_1^{(i)}, \ldots, p_{n_i}^{(i)}$ are distinct points, disjoint from $q^{(i)}$, that include the attachment points of $E_i$ ``heading away" from $E_1$;
\end{enumerate}

Let $x_1, \ldots, x_m$ be the points $p_j^{(i)}$ belonging to only one of the $E_i$. Consider the following problem:

\begin{center}
\fbox{
\parbox{.8\textwidth}{
Given $a_1, \ldots, a_m \in k$, find a section $\sigma \in \Gamma(E, \mathcal{L})$ so that
\[
   \sigma(x_i) = a_i, \quad 1 \leq i \leq m.
\]
}
}
\end{center}

There is a codimension $g$ condition on the $a_i$s under which a solution exists. Moreover, if there is a solution, then there is a 1-dimensional family of solutions.

If $E$ has genus one, then there are constants $c_1, \ldots, c_m \in k$ so that a solution exists if and only if
\[
   c_1a_1 + \cdots + c_ma_m = 0,
\]
and $c_j = 0$ if and only if the path in the dual graph of $E$ from $\core(E)$ to $x_j$ includes either
\begin{enumerate}
  \item a rational component of $E_1$ or
  \item a rational component $Z$ of $E_i$ for $i \neq 1$ so that $Z$ does not contain $q^{(i)}$.
\end{enumerate}
\end{lemma}

We will prove Lemma \ref{lem:sectionvaluesinduction} by induction on $r$. To give a section
$\sigma \in \Gamma(E, \mathcal{L})$ it is equivalent to specify a tuple of sections
$(\sigma_i) \in \prod_{i = 1}^r \Gamma(E_i, \mathcal{L} \otimes \mathscr{O}_{E_i})$ whose values agree on the nodes joining the $E_i$.
Therefore we start by gathering some lemmas about the values of $\mathcal{L}$ on the $E_i$.

\begin{lemma} \label{lem:genus_zero_nodal_section_values}
Let $p_1, \ldots, p_n, q$ be $n + 1$ distinct smooth closed points of a connected reduced genus zero nodal curve $D$. Then
\begin{enumerate}
  \item For any $a_1, \ldots, a_n \in k$, there is a unique global section $\sigma$ of $\mathscr{O}_{D}(p_1 + \cdots + p_n - q)$ satisfying the equations
\[
  (\star) \quad \sigma(p_i) = a_i, \quad 1 \leq i \leq n.
\]
  \item There are constants $c_1, \ldots, c_n$, so that for all global sections $\sigma$ of $\mathscr{O}_{D}(p_1 + \cdots + p_n - q)$
\[
  \sigma(q) = c_1\sigma(p_1) + \cdots + c_n\sigma(p_n)
\]
Moreover, $c_i$ is nonzero if and only if $p_i$ is on the same component of $D$ as $q$.
\end{enumerate}
\end{lemma}

\begin{proof}
The lemma holds when $D$ is irreducible by identifying the sections of $\mathscr{O}_{D}(p_1 + \cdots + p_n - q)$ with polynomials of degree $n - 1$ on $\mathbb{P}^1$. We proceed by induction on the number of irreducible components of $D$.

Assume that $D$ has more than one irreducible component. Let $T$ be the tree whose vertices are the irreducible components of $D$ with an
edge between those components that intersect. Root $T$ at the component containing $q$. Let $D'$ be a leaf of $T$, and let $D''$ be the union of the other irreducible components
of $D$. Let $P' = (p_1 + \cdots + p_n)|_{D'}$ and $P'' = (p_1 + \cdots + p_n)|_{D''}$. A global section of
$\mathscr{O}_D(p_1 + \cdots + p_n - q)$ on $D$ satisfying equations ($\star$) is the same as a pair of sections $\sigma'' \in \Gamma(D'', \mathscr{O}_{D''}(P'' - q))$ and $\sigma' \in \Gamma(D', \mathscr{O}_{D'}(P'))$ so that
\begin{enumerate}
  \item $\sigma''(p_i) = a_i$ when $p_i \in D''$,
  \item $\sigma'(p_i) = a_i$ when $p_i \in D'$, and
  \item $\sigma''(q') = \sigma'(q')$.
\end{enumerate}
By induction, there is a unique $\sigma''$ satisfying (i). So we look for a section $\sigma'$ satisfying (ii) and (iii). Consider the exact sequence
\[
  0 \to \mathscr{O}_{D'}(-q') \to \mathscr{O}_{D'}(P') \to k_{q'} \oplus \bigoplus_{p \in P'} k_p \to 0,
\]
where the last map is the direct sum of the evaluation homomorphisms at $p \in P'$ and $q'$. Since $D' \cong \mathbb{P}^1_k$ and $\mathscr{O}_{D'}(-q') \cong \mathscr{O}_{\mathbb{P}^1_k}(-1)$, taking global sections
yields an isomorphism $\Gamma(D', \mathscr{O}_{D'}(P')) \to k_q \oplus \bigoplus_{p \in P'} k_p$. It follows that there is a unique $\sigma'$ satisfying (ii) and (iii).

This leaves only the statement about whether coefficients are 0 or nonzero. We know that the statement holds for $c_i$ with $p_i \in D''$ by induction. For the other $c_i$s, we notice that $\sigma''$ was chosen before $\sigma'$, so the value of
$\sigma''$ at $q$ is not affected by the value of $\sigma'$ at $p_i \in D'$.
\end{proof}

\begin{lemma} \label{lem:connecting_hom_is_sum}
Suppose $D$ is a reduced nodal curve. Let $p_1, \ldots, p_n$ be distinct smooth closed points of $D$, $n \geq 1$.
For each $i$, $1 \leq i \leq n$, let
\[
  \delta_i : k(p_i) \to H^1(D, \mathscr{O}_D)
\]
be the connecting homomorphism associated to the short exact sequence
\[
  0 \to \mathscr{O}_D \to \mathscr{O}_D(p_i) \to k(p_i) \to 0.
\]

Then $\delta : \bigoplus_{i = 1}^n k(p_i) \to H^1(D, \mathscr{O}_D)$, the connecting homomorphism associated to the short exact sequence
\[
   0 \to \mathscr{O}_D \to \mathscr{O}_D(p_1 + \cdots + p_n) \to \bigoplus_{i = 1}^n k(p_i) \to 0,
\]
is the direct sum of the maps $\delta_i$.
\end{lemma}

\begin{proof}
Notice that for each $i$ there is a diagram of exact sequences
\[
  \xymatrix{
    0 \ar[r] & \mathscr{O}_D \ar[r] \ar@{=}[d] & \mathscr{O}_D(p_i) \ar[d] \ar[r] & k(p_i) \ar[r] \ar[d] & 0 \\
    0 \ar[r] & \mathscr{O}_D \ar[r] & \mathscr{O}_D(p_1 + \cdots + p_n) \ar[r] & \bigoplus_{j = 1}^n k(p_j) \ar[r] &  0
  }
\]
in which the rightmost map is the inclusion into the $i$th coordinate. Taking associated long exact sequences, we arrive at the commutative square
\[
  \xymatrix{
     \cdots \ar[r] & k(p_i) \ar[d] \ar[r]^{\delta_i} & H^1(D, \mathscr{O}_D) \ar@{=}[d] \ar[r] & \cdots \\
     \cdots \ar[r] & \bigoplus_{j = 1}^n k(p_j) \ar[r]^{\delta} & H^1(D, \mathscr{O}_D) \ar[r] & \cdots.
  }
\]
This implies the result.
\end{proof}

We say that a nodal curve $D$ of genus one is \emph{minimal} if it is equal to its core. Equivalently $D$ is minimal if and only if $D$ is a smooth
elliptic curve or a ring of nodally attached rational curves.

\begin{lemma} \label{lem:elliptic_core_no_h1}
Suppose $D$ is a minimal genus one curve. Let $Z$ be an effective divisor of positive degree on $D$. Then $H^1(D, \mathscr{O}_D(Z)) = 0$.
\end{lemma}
\begin{proof}
In any case, $D$ is Gorenstein with dualizing sheaf $\omega_D \cong \mathscr{O}_D$. By Serre duality,
\[
  H^1(D, \mathscr{O}_D(Z)) \cong H^0(D, \mathscr{O}_D(-Z))^\vee.
\]
If $D$ is smooth elliptic, the latter group is zero since $\mathscr{O}_D(-Z)$ has negative degree.

If $D$ is a ring of rational curves, let $D_1, \ldots, D_l$ be its components, and let $p_1, \ldots, p_l$ be its nodes, so that $p_i$ joins $D_i$ to $D_{i + 1}$ for $i < l$ and $p_l$ joins $D_l$ to $D_1$.
Then we may identify $H^0(D, \mathscr{O}_D(Z))$ with the set of tuples of sections
$(\sigma_i) \in \prod_i \Gamma(D_i, \mathscr{O}_{D_i}(-Z|_{D_i}))$ satisfying $\sigma_i(p_i) = \sigma_{i + 1}(p_i)$ for all $1 \leq i < l$ and
$\sigma_l(p_l) = \sigma_1(p_l)$.

Now, each $\mathscr{O}_{D_i}(-Z|_{D_i})$ has non-positive degree. For those $i$ where $-Z|_{D_i} = 0$, if for either point $p_j \in D_i$ we have $\sigma_i(p_j) = 0$, then $\sigma_i$ is zero. For those $i$ where $\mathscr{O}_{D_i}(-Z|_{D_i})$ has negative degree, $\sigma_i = 0$. Since there is at least one $i$ with $\deg(-Z|_{D_i}) < 0$, it follows by working our way around the ring that the only global section of $\mathscr{O}_D(-Z)$ is zero.
\end{proof}

\begin{lemma} \label{lem:arbitrary_genus_section_values}
Let $D$ be a connected reduced nodal curve over $k$ of genus $g$. Let $p_1, \ldots, p_n$ be distinct smooth closed points of $D$, $n \geq 1$, so that $H^1(\mathscr{O}_D(p_1 + \cdots + p_n)) = 0$. Consider the following problem:

\medskip

\begin{center}
\fbox{
\parbox{.8\textwidth}{
Given $a_1, \ldots, a_n \in k$, find a global section $\sigma$ of $\mathscr{O}_D(p_1 + \cdots + p_n)$ so that
\[
   \sigma(p_i) = a_i, \quad 1 \leq i \leq n.
\]
}
}
\end{center}

\medskip

There is a codimension $g$ condition on the $a_i$s under which a solution exists. Moreover, if there is a solution, then there is a 1-dimensional family of solutions.

When $D$ has genus 1, there are constants $c_1, \ldots, c_n$ so that the problem admits a solution if and only if
\[
  c_1a_1 + \cdots + c_na_n = 0,
\]
and the coefficient $c_i$ is nonzero if and only if $p_i$ lies in the minimal genus 1 subcurve of $D$.
\end{lemma}

\begin{proof}
There is a short exact sequence
\[
  0 \to \mathscr{O}_D \to \mathscr{O}_D(p_1 + \cdots + p_n) \overset{\mathrm{ev}}{\longrightarrow} \bigoplus_{i = 1}^n k(p_i) \to 0
\]
where the right map is induced by evaluation at $p_1, \ldots, p_n$.

Taking global sections, we obtain
\begin{equation} \label{eq:top_lifting_condition}
  0 \to k \to \Gamma(D, \mathscr{O}_D(p_1 + \cdots + p_n)) \overset{\mathrm{ev}}{\longrightarrow} k^n \to k^g \to 0.
\end{equation}
Our problem can now be identified with the problem
of finding a lift of $(a_1, \ldots, a_n) \in k^n$ to $\Gamma(D, \mathscr{O}_D(p_1 + \cdots + p_n))$. By exactness, there is a codimension $g$ condition under which
a lift exists. Exactness also implies that solutions form a 1-dimensional family.

\medskip

Now suppose that $D$ is a \emph{minimal} genus 1 curve. Then by Lemma \ref{lem:elliptic_core_no_h1}, $H^1(D, \mathscr{O}_{D}(p_i)) = 0$ for each $i$, so the map $\delta_i : k(p_i) \to H^1(D, \mathscr{O}_{D}) \cong k$
associated to the short exact sequence
\[
  0 \to \mathscr{O}_{D} \to \mathscr{O}_{D}(p_i) \to k(p_i) \to 0
\]
is surjective. In particular, there is a nonzero constant $c_i$ so that $\delta_i: a \mapsto c_ia$. By Lemma \ref{lem:connecting_hom_is_sum},
the map $k^n \to k^g = k$ in the exact sequence (\ref{eq:top_lifting_condition}) is $(a_1, \ldots, a_n) \mapsto c_1a_1 + \cdots + c_na_n$.

\medskip

Finally, suppose that $D$ has genus 1, but is not minimal. Decompose $D$ into connected subcurves $D_1, \ldots, D_l$ where $D_1$ is the core of $D$ and the remaining $D_j$, $j = 2, \ldots, l$, are the connected components of $\ol{D - D_1}$. Note that for $j \neq 1$, $D_j$ meets $D_1$ at a unique point $q_j$.
By Lemma \ref{lem:connecting_hom_is_sum}, the map $k^n \to k^g = k$ in the exact sequence (\ref{eq:top_lifting_condition}) is of the form $(a_1, \ldots, a_n) \mapsto c_1a_1 + \cdots + c_na_n$.
We just need to determine which $c_i$ are nonzero.

For each $j$, let $P_j$ be the set of points $p_i$ so that $p_i \in D_j$. A global section $\sigma \in \Gamma(D, \mathscr{O}_D(p_1 + \cdots + p_n))$
is equivalent to a tuple of sections $(\sigma_j) \in \prod_{j = 1}^l \Gamma(D_j, \mathscr{O}_{D_j}(P_j))$, agreeing at the points $q_j$.
By the result for minimal genus 1 curves, we can find a section $\sigma_1 \in \Gamma(D_1, \mathscr{O}_{D_1}(P_1))$ with $\sigma_1(p_i) = a_i$ for the $p_i \in D_1$ if and only if
\[
  \sum_{i} c_ia_i = 0
\]
where the sum is quantified over those $i$ so that $p_i \in D_1$. Moreover, for these $i$ we have that $c_i \neq 0$. Taking $c_i = 0$ for those $i$ with $p_i \not\in D_1$, we can rewrite the condition for the existence of $\sigma_1$ as
\[
  \sum_{i = 1}^n c_ia_i = 0.
\]

To complete the proof, we need to show that for any such $\sigma_1$, there is a unique extension of $\sigma_1$ to $D$ satisfying $\sigma(p_i) = a_i$ for
all $i$. To see this, for each $j = 2, \ldots, l$, build a sequence
\[
  0 \to \mathscr{O}_{D_j}(-q_1) \to \mathscr{O}_{D_j}(P_j) \to k(q_j) \oplus \bigoplus_{p \in P_J} k(p) \to 0
\]
then apply the global sections functor.
\end{proof}

\begin{proof}[Proof of Lemma \ref{lem:sectionvaluesinduction}]
Lemma \ref{lem:arbitrary_genus_section_values} is the result when $r = 1$. We proceed by induction on $r$.

Let $r > 1$. Let $T$ be the graph with vertices $E_1, \ldots, E_r$ and an edge between $E_{i_1}$ and $E_{i_2}$ for $i_1 \neq i_2$ whenever
$E_{i_1} \cap E_{i_2}$ is non-empty. Since $E_1$ contains the core of $E$, $T$ must be a tree. Consider $T$ as a rooted tree rooted at $E_1$.

Choose $j$ so that $E_j$ is one of the leaves of $T$. Note that $j \neq 1$ and $p^{(j)}_1, \ldots, p^{(j)}_{n_j}$
all lie in the set $\{ x_1, \ldots, x_m \}$. Without loss of generality, we may assume $p^{(j)}_1 = x_1, \ldots, p^{(j)}_{n_j} = x_{n_j}$.
Let $E_{\neq j} := \bigcup_{i \neq j} E_i$.

Now, the elements $\sigma \in \Gamma(E, \mathcal{L})$ satisfying $\sigma(x_i) = a_i$ for $1 \leq i \leq m$ are in bijection with
pairs $(\sigma_j, \sigma_{\neq j}) \in \Gamma(E_j, \mathcal{L} \otimes \mathscr{O}_{E_j}) \times \Gamma(E_{\neq j}, \mathcal{L} \otimes \mathscr{O}_{E_{\neq j}})$ such that
\begin{enumerate}
  \item $\sigma_{j}(x_i) = a_i$ for $1 \leq i \leq n_j$;
  \item $\sigma_{\neq j}(x_i) = a_i$ for $n_j < i \leq m$;
  \item $\sigma_{j}(q^{(j)}) = \sigma_{\neq j}(q^{(j)})$.
\end{enumerate}

By Lemma \ref{lem:genus_zero_nodal_section_values}, there is a unique section $\sigma_{j}$ satisfying (i). Let $\sigma_{j}$ be this section. Let
$c_1, \ldots, c_{n_j}$ be the constants of the lemma, so that in particular we have
\[
  \sigma_{j}(q^{(j)}) = c_1a_1 + \cdots + c_{n_j}a_{n_j}.
\]
Recall that $c_i \neq 0$ if and only if $x_i$ lies the irreducible component of $E_j$ containing $q^{(j)}$. Since $\mathcal{L}$ has non-negative
degree on each irreducible component of $E$, at least one of the $c_i$s must be nonzero.

We conclude that the problem of finding $\sigma \in \Gamma(E, \mathscr{L})$ satisfying $\sigma(x_i) = a_i$ for $1 \leq i \leq m$ is equivalent
to finding $\sigma_{\neq j} \in \Gamma(E_{\neq j}, \mathscr{L} \otimes \mathscr{O}_{E_{\neq j}})$ satisfying $\sigma_{\neq j}(x_i) = a_i)$ for $n_j < i \leq m$
and $\sigma_{\neq j}(q^{(j)}) = c_1a_1 + \cdots + c_{n_j}a_{n_j}$. By induction, there is a codimension $g$ condition on
\[
  (c_1a_1 + \cdots + c_{n_j}a_{n_j}, a_{n_j + 1}, \ldots, a_m)
\]
under which a solution exists, and when a solution exists there is a 1-dimensional family of solutions. Since pullback along a surjective linear map preserves codimension, this codimension $g$ condition pulls back to a codimension $g$ condition on $(a_1, \ldots, a_m).$

All that remains is the additional claim when $E$ has genus one, so assume that $E$ has genus one. By induction, the codimension one condition on the existence
of $\sigma$ can be written
\begin{align*}
  c_1'(c_1a_1 + \cdots + c_{n_j}a_{n_j}) + c_{n_j + 1}'a_{n_j + 1} + \cdots + c_m'a_m = \\
  c_1'c_1a_1 + \cdots + c_1'c_{n_j}a_{n_j} + c_{n_j + 1}'a_{n_j + 1} + \cdots + c_m'a_m  = 0
\end{align*}
for some constants $c_1', c_{n_j + 1}', \ldots, c_m'$ where $c_1'$ is zero if and only if the path from the core to $q^{(j)}$ passes through
a rational component of $E_1$ or a component of an $E_i$ not containing $q^{(i)}$, and similarly for the other $c_i'$s. Comparing when $c_1'$ and $c_1, \ldots, c_{n_j}$
are zero, we conclude that the coefficients of the $a_i$s are zero when claimed.
\end{proof}

Shuffling our assumptions around slightly, the same methods can be used to give a convenient criterion for acyclicity. In particular, we see
that a steep mesa with genus one support is automatically acyclic.

\begin{lemma} \label{lem:acyclicity_criterion}
Assume we are in the situation of Lemma \ref{lem:sectionvaluesinduction}, except that we do not assume $H^1(E, \mathcal{L}) = 0$.
Then $H^1(E, \mathcal{L}) = 0$ if and only if $H^1(E_1, \mathcal{L}) = 0$.

In particular, if $S$ is the spectrum of an algebraically closed field and $(\pi : C \to S, \ol{\lambda})$ satisfies the conditions of a mesa curve except
possibly acyclicity, then $(\pi : C \to S, \ol{\lambda})$ is a mesa curve if and only if $H^1(E_1, \mathscr{O}_{E_1}(-\ol{\lambda})) = 0$ for
$E_1$ equal to the top of $\ol{\lambda}$.
\end{lemma}

\begin{proof}
The claim is trivial when $r = 1$.

If $r > 1$, choose $E_j$ and $E_{\neq j}$ as in the proof of Lemma \ref{lem:sectionvaluesinduction}. Then
\[
  0 \to \mathcal{L} \to (\mathcal{L} \otimes \mathscr{O}_{E_{\neq j}}) \oplus (\mathcal{L} \otimes \mathscr{O}_{E_j}) \to k(q^{(j)}) \to 0
\]
is exact.

The long exact sequence in cohomology implies
\[
  H^1(E, \mathcal{L}) \to H^1(E_{\neq j}, \mathcal{L} \otimes \mathscr{O}_{E_{\neq j}}) \oplus H^1(E_{j}, \mathcal{L} \otimes \mathscr{O}_{E_{j}}) \to 0
\]
is exact. In particular, if $H^1(E, \mathcal{L}) = 0$, then $H^1(E_{\neq j}, \mathcal{L} \otimes \mathscr{O}_{E_{\neq j}}) = 0$, so
$H^1(E_1, \mathcal{L} \otimes \mathscr{O}_{E_1}) = 0$ by induction.

Conversely, if $H^1(E_1, \mathcal{L} \otimes \mathscr{O}_{E_1}) = 0$, then $H^1(E_{\neq j}, \mathcal{L} \otimes \mathscr{O}_{E_{\neq j}}) = 0$ by induction. Since $\mathcal{L}$ has non-negative degree on all components, it is easy to check that $H^1(E_j, \mathcal{L} \otimes \mathscr{O}_{E_j}) = 0$ as well. The long exact sequence in cohomology now yields
\[
  \Gamma(E_{\neq j}, \mathcal{L} \otimes \mathscr{O}_{E_{\neq j}}) \oplus \Gamma(E_j, \mathcal{L} \otimes \mathscr{O}_{E_j}) \to k(q^{(j)}) \to H^1(E, \mathscr{L}) \to 0.
\]
Applying Lemma \ref{lem:genus_zero_nodal_section_values} to $E_j$ implies that the first map is surjective, and we are done.
\end{proof}

\subsection{Flatness and commutativity with base change}

In this section we show that $\ol{U}$ is flat over $S$ and the formation of $\ol{U}$ commutes with base change in $S$. It is interesting how
closely related the two questions are.

\begin{lemma} \label{lem:truncate_flat_complex}
Let
\[
  C^\bullet : 0 \to C^0 \to \cdots \to C^n \to 0
\]
be a cochain complex of flat $A$-modules so that $H^i(C^\bullet) = 0$ for $i > j$. Then $C^\bullet$ is quasi-isomorphic
to its truncation $\tau_{\leq j}C^\bullet$, and $\tau_{\leq j}C^\bullet$ is a cochain complex of flat $A$-modules.
\end{lemma}

\begin{proof}
Omitted.
\end{proof}

\begin{theorem} \label{thm:pushforward_flat_commutes_with_base_change}
Let
\[
  \xymatrix{
    U_T \ar[r]^g \ar[d]_{\pi_T} & U \ar[d]^{\pi} \\
    T \ar[r]^f  & S
  }
\]
be a cartesian diagram of schemes with $\pi$ quasicompact and separated, and let $\mathscr{F}$ be a quasicoherent sheaf on $U$, flat over $S$. Then
\begin{enumerate}
  \item If $R^q\pi_*\mathscr{F} = 0$, for $q > i$, then the natural map $f^*(R^i\pi_*\mathscr{F}) \to R^i\pi_{T, *}(g^*\mathscr{F})$ is an isomorphism.
  \item If $R^q\pi_*\mathscr{F} = 0$ for $q > 0$, then $\pi_*\mathscr{F}$ is flat over $S$.
\end{enumerate}
\end{theorem}

\begin{proof}

The assertion is local on $S$ and $T$, so we may assume that $S = \Spec A$, $T = \Spec B$. Then what we have to show is
\begin{enumerate}
  \item If $H^q(U, \mathscr{F}) = 0$ for $q > i$, then $H^i(U, \mathscr{F}) \otimes_A B \to H^i(U_T, \mathscr{F} \otimes_A B)$ is an isomorphism.
  \item If $H^q(U, \mathscr{F}) = 0$ for $q > 0$, then $H^0(U, \mathscr{F})$ is a flat $A$-module. 
\end{enumerate}

Since $\pi$ is quasicompact and separated, we can find a cover $\mathfrak{U} = \{ U_1, \ldots, U_n \}$ of $U$ by open affines with affine intersections. Then $H^i(U, \mathscr{F})$ is computed by the \Cech complex
\[
  \check{C}^p(\mathscr{F}, \mathfrak{U}) = \bigoplus_{i_1 < i_2 < \cdots < i_p} \Gamma(U_{i_1} \cap \cdots \cap U_{i_p}, \mathscr{F})
\]
in the sense that $H^i(\check{C}^\bullet(\mathfrak{U}, \mathscr{F})) \cong H^i(U, \mathscr{F})$. We note that since $\mathscr{F}$ is $S$-flat, $\check{C}(\mathfrak{U}, \mathscr{F})$ is a complex of flat $A$-modules.

Since $f$ is affine, the pullback cover $\mathfrak{U} \times_S T = \{ U_1 \times_S T, \ldots, U_n \times_S T \}$ of $U_T$ is again an open affine cover with affine intersections, so $H^i(U_T, \mathscr{F} \otimes_A B)$ is computed
by the complex $\check{C}(\mathfrak{U} \times_S T, \mathscr{F} \otimes_A B) \cong \check{C}(\mathfrak{U}, \mathscr{F}) \otimes_A B$.

Fix an integer $i \geq 0$, and assume $H^q(U, \mathscr{F}) = 0$ for $q > i$. By Lemma \ref{lem:truncate_flat_complex}, there is a subcomplex $C^\bullet$ of $\check{C}(\mathfrak{U}, \mathscr{F})$ with amplitude in $[0, i]$ so that the inclusion
is a quasi-isomorphism and each $C^p$ is a flat $A$-module, $p = 0, \ldots, i$. Since $C^\bullet$ and $\check{C}(\mathfrak{U}, \mathscr{F})$ are quasi-isomorphic complexes of flat $A$-modules, $C^\bullet \otimes_A B$ is quasi-isomorphic
to $\check{C}(\mathfrak{U}, \mathscr{F}) \otimes_A B \cong \check{C}(\mathfrak{U} \times_S T, \mathscr{F} \otimes_A B)$. Since $- \otimes_A B$ is right exact,
\begin{align*}
   H^i(U, \mathscr{F}) \otimes B &= H^i(C^\bullet) \otimes_A B \\
     &\cong H^i(C^\bullet \otimes_A B) \\
     &\cong H^i(\check{C}(\mathfrak{U} \times_S T, \mathscr{F} \otimes_A B)) \\
     &\cong H^i(U_T, \mathscr{F} \otimes_A B),
\end{align*}
which establishes (i).

For (ii), we note that if $i = 0$, $H^0(U, \mathscr{F}) = H^0(C^\bullet) = C^0$ is $A$-flat.
\end{proof}

\begin{lemma} \label{lem:2g_minus_2_plus_s}
Suppose $C$ is a smooth curve of genus $g$. Let $p_1, \ldots, p_s$ be distinct closed points of $C$, and let $\mathscr{L}$ be a locally free sheaf of rank one on $C$ with $\mathrm{deg}(\mathscr{L}) > 2g - 2 + s$.
Then the evaluation map $\Gamma(C, \mathscr{L}) \to \bigoplus_{i = 1}^s k(p_s)$ is surjective.
\end{lemma}
\begin{proof}
We have an exact sequence
\[
  0 \to \mathscr{L}(-p_1 - \cdots - p_s) \to \mathscr{L} \to \bigoplus_{i = 1}^s k(p_i) \to 0.
\]
Taking the associated long exact sequence and noting $H^1(C, \mathscr{L}(-p_1 - \cdots - p_s)) = 0$ gives the result.
\end{proof}

\begin{proposition} \label{prop:no_coh_minus_lambda_bar}
Assume that we are in the standard situation. For $q \geq 1$,
\[
  H^q(U, \mathscr{O}_U(-\ol{\lambda})) = 0.
\]
\end{proposition}
\begin{proof}
Let $\Sigma$ be the sum of the divisors of the sections $\sigma_i$.
Note that
\[
  \mathscr{O}_U(-\ol{\lambda}) \cong \colim_{n \in \mathbb{N}} \mathscr{O}_C(-\ol{\lambda} + n\Sigma).
\]
Affine locally this is the isomorphism
$\Gamma(C, \mathscr{O}_C)[f^{-1}] \cong \colim_{n \in \mathbb{N}} f^{-n}\Gamma(C, \mathscr{O}_C)$
where $f$ is a local equation for $\Sigma$.

Since cohomology commutes with colimits, it is enough to show that $H^q(C, \mathscr{O}_C(-\ol{\lambda} + n\Sigma)) = 0$ for all sufficiently large $n$.
We have assumed that $S$ is finite dimensional; suppose it has dimension $d$. Then $C$ is $d + 1$ dimensional, so $H^{d+2}(C, \mathscr{O}_C(-\ol{\lambda} + n\Sigma)) = 0$.
We now work by descending induction on $q$.

Suppose $H^{q + 1}(C, \mathscr{O}_C(-\ol{\lambda} + n\Sigma)) = 0$ for all sufficiently large $n$. Then by Theorem \ref{thm:pushforward_flat_commutes_with_base_change}, cohomology commutes with base change in degree $q$
for $\mathscr{O}_C(-\ol{\lambda} + n\Sigma)$. Since $\pi$ is proper and $\mathscr{O}_C(-\ol{\lambda} + n\Sigma)$ is coherent, $H^q(C, \mathscr{O}_C(-\ol{\lambda} + n\Sigma))$ is finitely generated. Let $\ol{t}$ be any geometric point of $S$,
and let $C_{\ol{t}}, E_{\ol{t}}$ be the pullbacks of $C$ and $E$ to $\ol{t}$ respectively. By Nakayama's lemma, it is enough to show that
\[
  H^q(C_{\ol{t}}, \mathscr{O}_{C_{\ol{t}}}(-\ol{\lambda} + n\Sigma)) = 0.
\]
for all $n \geq N$, where $N$ is independent of $\ol{t}$.

If $q \geq 2$, then $H^q(C_{\ol{t}}, \mathscr{O}_{C_{\ol{t}}}(-\ol{\lambda} + n\Sigma))$ vanishes since $C_{\ol{t}}$ has dimension 1.

If $q = 1$, let $Z_1, \ldots, Z_r$ be the components of $C_{\ol{t}}$ not contained in $E_{\ol{t}}$, and let $p_1, \ldots, p_s$ be the intersection points of the $Z_i$s with each other and with $E_{\ol{t}}$. Consider the short exact sequence
\[
  0 \to \mathscr{O}_{C_{\ol{t}}} \to \mathscr{O}_{E_{\ol{t}}} \oplus \bigoplus_{i = 1}^r \mathscr{O}_{Z_i} \to \bigoplus_{i = 1}^s k(p_i) \to 0.
\]
Twisting by $-\ol{\lambda} + n\Sigma$ and then taking the associated long exact sequence gives
\[
  H^0(E_{\ol{t}}, \mathscr{O}_{E_{\ol{t}}}(-\ol{\lambda})) \oplus \bigoplus_{i = 1}^r H^0(Z_i, \mathscr{O}_{Z_i}(-\ol{\lambda} + n\Sigma|_{Z_i})) \to k^s \to H^1(C_{\ol{t}}, \mathscr{O}_{C_{\ol{t}}}(-\ol{\lambda} + n\Sigma)).
\]
By Lemma \ref{lem:etale_local_reductions}(iv), $\mathscr{O}_{Z_i}(-\ol{\lambda} + n\Sigma|_{Z_i})$ has degree at least $n$. Applying Lemma \ref{lem:2g_minus_2_plus_s} to each $Z_i$, we can make the first map surjective by choosing $n$ greater than $2g - 2 + s$.
Looking a little further to the right in the same long exact sequence we now have
\[
  0 \to H^1(C_{\ol{t}}, \mathscr{O}_{C_{\ol{t}}}(-\ol{\lambda} + n\Sigma)) \to H^1(E_{\ol{t}}, \mathscr{O}_{E_{\ol{t}}}(-\ol{\lambda})) \oplus \bigoplus_{i = 1}^m H^1(Z_i, \mathscr{O}_{Z_i}(-\ol{\lambda} + n\Sigma|_{Z_i}))
\]
Choosing $n$ larger than $2g - 2 + s$, we can also make $H^1(Z_i, \mathscr{O}_{Z_i}(- \ol{\lambda} + n\Sigma|_{Z_i}))$ vanish.
By acyclicity, $H^1(E_{\ol{t}}, \mathscr{O}_{E_{\ol{t}}}(-\ol{\lambda}))$ is also zero, so $H^1(C_{\ol{t}}, \mathscr{O}_{C_{\ol{t}}}(-\ol{\lambda} + n\Sigma)) = 0$. The number of nodes in a fiber is bounded above, so we may choose a lower bound on $n$ uniformly for all $\ol{t}$, and we are done.
\end{proof}

\begin{proposition} \label{prop:bbar_sequence_is_nice}
Assume we are in the standard situation. Then
\begin{enumerate}
  \item the $\ol{B}$-sequence
  \[
   0 \to \Gamma(S, \mathscr{O}_S(-\rho)) \to B(U) \to \ol{B}(U) \to 0
  \]
  is short exact.
  \item The terms of the $\ol{B}$-sequence are each flat over $S$.
  \item If $f : T = \Spec R \to S$ is a morphism of affine Noetherian, finite-dimensional schemes, we may pull back the data of the standard
    situation to obtain $\pi_T: C_T \to T$, $\ol{\lambda}_T$, $\rho_T$, sections $\sigma_{T,1}, \ldots, \sigma_{T, h}$, and their complement $U_T$. These data
   are still in the standard situation, so we may form $\ol{B}(U_T)$.  Then there is a natural isomorphism of short exact sequences of $R$-modules:
  \begin{equation} \label{eq:bbar_sequence_pullback}
    \xymatrix{
      0 \ar[r] & \Gamma(S, \mathscr{O}_S(-\rho)) \otimes_A R \ar[r] \ar[d] & B(U) \otimes_A R \ar[r] \ar[d] & \ol{B}(U) \otimes_A R \ar[r] \ar[d] & 0 \\
      0 \ar[r] & \Gamma(T, \mathscr{O}_T(-\rho_T)) \ar[r] & B(U_T) \ar[r] & \ol{B}(U_T) \ar[r] & 0.
    }
  \end{equation}
\end{enumerate}
\end{proposition}
\begin{proof}
Let us begin with (i). We have to show that the map $\Gamma(S, \mathscr{O}_S(-\rho)) \to B(U)$ is injective; the rest of the sequence
is exact by definition. Recall that $B(U) = \Gamma(U, \mathscr{O}_C(-\ol{\lambda})) \oplus \Gamma(S, \mathscr{O}_S)$ and the first map
$\Gamma(S, \mathscr{O}_S(-\rho)) \to B(U)$ is given by $d \mapsto ( (\rho - \ol{\lambda})(d), -\rho(d) )$.

We always have (iii) for flat maps $T \to S$, so we may work \'etale locally on $S$. Therefore we may assume that there is a section
$\sigma : S \to C$ of $\pi$ so that $\sigma$ goes through the locus in $C$ where $\ol{\lambda} = \rho$. (For each geometric point $\ol{s}$ of $S$,
choose a neighborhood with a uniform set of charts, then, localizing if necessary, choose a section through some point
of the core of $E$ in the fiber over $\ol{s}$. This section will have the required property.)

The map $\sigma$ is in fact a map of log schemes, since the log structure of $C$ along the image of $\sigma$ is that pulled back from
$S$. Now, consider the composite
\[
  \pi^{-1}\mathscr{O}_S(-\rho) \to \mathscr{O}_C(-\rho) \overset{\rho - \ol{\lambda}}{\to} \mathscr{O}_C(-\ol{\lambda}).
\]
Taking global sections on $U$ and considering the pullback maps induced by $\sigma$ gives us a commutative square
\[
  \xymatrix{
   \Gamma(U, \pi^{-1}\mathscr{O}_S(-\rho)) \ar[r] \ar[d]_{\sigma^*} & \Gamma(U, \mathscr{O}_C(-\ol{\lambda})) \ar[d]^{\sigma^*} \\
   \Gamma(S, \mathscr{O}_S(-\rho)) \ar@{=}[r]  & \Gamma(S, \mathscr{O}_S(-\rho)).
   }
\]
where the left vertical arrow is the canonical isomorphism, and the lower horizontal arrow is the identity, since it is induced by the zero element
of $\ol{M}_S$. The composite map from the lower left to upper left to upper right corner is the map
from $\Gamma(S, \mathscr{O}_S(-\rho))$ to the first coordinate of $B(U)$. Therefore $(f,c) \mapsto \sigma^*(f)$ is a splitting of the
$\ol{B}$-sequence, which implies (i).

Now consider (iii). We have natural isomorphisms $\Gamma(S, \mathscr{O}_S(-\rho)) \otimes_A R \to \Gamma(S, \mathscr{O}_T(-\rho_T))$
and $\Gamma(S, \mathscr{O}_S) \otimes_A R \to \Gamma(S, \mathscr{O}_S)$. By Proposition \ref{prop:no_coh_minus_lambda_bar} and
Theorem \ref{thm:pushforward_flat_commutes_with_base_change}, the base change map
$\Gamma(U, \mathscr{O}_C(-\ol{\lambda})) \otimes_A R \to \Gamma(U_T, \mathscr{O}_{C_T}(-\ol{\lambda}_T))$ is also an isomorphism.
This gives us a solid diagram
\[
  \xymatrix{
   0 \ar@{-->}[r] & \Gamma(S, \mathscr{O}_S(-\rho)) \otimes_A R \ar[r] \ar[d] & B(U) \otimes_A R \ar[r] \ar[d] & \ol{B}(U) \otimes_A R \ar[r] \ar@{-->}[d] & 0 \\
   0 \ar[r] & \Gamma(T, \mathscr{O}_T(-\rho_T)) \ar[r] & B(U_T) \ar[r] & \ol{B}(U_T) \ar[r] & 0
  }
\]
in which the solid vertical arrows are isomorphisms and each row is exact. The solid square commutes since the maps between sheaves associated to sections of $\ol{M}_S$
commute with pullback. The universal property of the cokernel gives us a unique isomorphism $\ol{B}(U) \otimes_A R \to \ol{B}(U_T)$ so that the diagram remains commutative. Since all vertical
arrows are isomorphisms, we may fill in a zero in the upper left. This proves (iii).

This leaves (ii). It is clear that $\Gamma(S, \mathscr{O}_S(-\rho))$ and $\Gamma(S, \mathscr{O}_S)$ are each flat over $A$.
By Proposition \ref{prop:no_coh_minus_lambda_bar} and Theorem \ref{thm:pushforward_flat_commutes_with_base_change},
$\Gamma(U, \mathscr{O}_C(-\ol{\lambda}))$ is also $A$-flat, so the $\ol{B}$-sequence is a flat resolution of $\ol{B}(U)$.
Now, the possible rings $R$ in (iii) include the rings $A/I$ for all ideals $I$ of $A$, so the top row of diagram (\ref{eq:bbar_sequence_pullback})
shows $\Tor^A_1(A/I, \ol{B}(U)) = 0$ for all $I$. It follows $\ol{B}(U)$ is flat over $A$, which concludes the proof.
\end{proof}

\begin{corollary}
In the standard situation, $\ol{\pi} : \ol{U} \to S$ is flat.
\end{corollary}

\begin{corollary} \label{cor:U_bar_commutes_base_change}
Assume we are in the standard situation and let $f : T = \Spec R \to S$ be a morphism of affine, Noetherian, finite-dimensional schemes. Pull
back the
data of the standard situation along $f$ to obtain $\pi_T : C_T \to T$, $U_T$, $E_T$, $\rho_T$, and $\sigma_{T,1}, \ldots, \sigma_{T, h}$ over
$T$. These data are again in the standard situation, so we may apply Definition \ref{def:contraction} to obtain a triangle
\[
  \xymatrix{
    U_T \ar[r] \ar[rd] & \ol{U_T} \ar[d] \\
    & T.
  }
\]

Then there is a uniquely specified isomorphism $\phi_{S,T} : \ol{U_T} \to \ol{U} \times_S T$ so that
\[
  \xymatrix{
    U_T \ar[rdd] \ar[rr] \ar[rd]& & \ol{U} \times_S T  \ar[ldd] \\
     & \ol{U_T} \ar[d] \ar^{\phi_{S,T}}[ru]\\
    & T
  }
\]
commutes. Moreover, these isomorphisms satisfy the cocycle condition in the sense that if $g : W \to T$ is a second
morphism of affine, Noetherian, finite-dimensional schemes,
\[
  \phi_{S,W} = \phi_{T,W} \circ \phi_{S,T}|_{W}.
\]
\end{corollary}
\begin{proof}
We get the isomorphism $\phi_{S,T}$ by taking Spec of the map $\ol{B}(U) \otimes_A R \to \ol{B}(U_T)$ from
Proposition \ref{prop:bbar_sequence_is_nice}(iii). The commutativity of $\phi_{S,T}$ with the map from $U$ follows from a straightforward
diagram chase. The commutativity of $\phi_{S,T}$ with the maps to $T$ holds since $\ol{B}(U) \otimes_A R \to \ol{B}(U_T)$ is a map
of $R$-modules. The cocycle condition holds since it holds of the cohomology--and--base change map.
\end{proof}

\subsection{Well-definedness of the contracted curve}

We will assume throughout the remainder of the section that we are in the standard situation.

Our first task will be to see that $\tau : U \to \ol{U}$ restricts to an isomorphism in the complement of $E$. This will ensure that the pushout defining $\ol{C}$ exists and that $\tau : C - E \to \ol{C} - \tau(E)$ is an isomorphism.

\begin{lemma} \label{lem:expression_for_I_tau_E}
Let $I_{\tau(E)} \subseteq \ol{B}(U)$ be the ideal of the scheme theoretic image of $E$ under $\tau$.
Let $\phi : B(U) \to \ol{B}(U)$ be the quotient map.
Then
\[
  \phi^{-1}(I_{\tau(E)}) = \Big\{ (f, c) \, \, \Big| \, \, f \in \Gamma(U, \mathscr{O}_U(-\ol{\lambda})) \text{ and }c \in \rho(\Gamma(S, \mathscr{O}_S(-\rho))) \Big\}
\]
and
\[
  I_{\tau(E)} = \Big\{ \phi(f, 0) \, \, \Big| \, \, f \in \Gamma(U, \mathscr{O}_U(-\ol{\lambda})) \Big\}.
\]
\end{lemma}
\begin{proof}
Let $\phi(f,c)$ be an arbitrary element of $\ol{B}(U)$. By definition of $\tau$, the image of $\phi(f,c)$ in $\Gamma(U, \mathscr{O}_U)$ is $\ol{\lambda}(f) + \pi^*c$.
By definition of $E$, $\ol{\lambda}(f)$ vanishes on $E$, so $\ol{\lambda}(f) + \pi^*c$ vanishes on $E$ if and only if $\pi^*c$ does. We now want to show that $\pi^*c$ vanishes on $E$ if and only if $c \in \rho(\Gamma(S, \mathscr{O}_S(-\rho)))$.

We may check this \'etale locally on $S$, so we may assume that there is a section $\sigma : S \to C$ going through the locus where $\ol{\lambda} = \rho$. Note $\sigma^*(\ol{\lambda}) = \rho$. Now, if $\pi^*c$ reduces to $0$ in $\mathscr{O}_E$, there is an open cover
$\{ U_i \}_{i \in I}$ of $U$ and elements $d_i \in \Gamma(U_i, \mathscr{O}_C(-\ol{\lambda}))$ so that $\pi^*c|_{U_i} = \ol{\lambda}(d_i)$. Pulling back along $\sigma$, we obtain an open cover $\{ V_i \}_{i \in I} = \{ \sigma^{-1}(U_i) \}_{i \in I}$ of $S$ and elements
$\sigma^*d_i \in \Gamma(V_i, \mathscr{O}_S(-\rho))$ so that 
\[
  c|_{V_i}  = \sigma^*(\pi^*c|_{U_i}) = \rho(\sigma^*d_i).
\]
That is, $c$ is locally in the image of $\mathscr{O}_S(-\rho) \overset{\rho}{\to} \mathscr{O}_S$. Since $S$ is affine, $c \in \rho(\Gamma(S, \mathscr{O}_S(-\rho)))$. Conversely, if $c \in \rho(\Gamma(S, \mathscr{O}_S(-\rho)))$, then since $\rho = \ol{\lambda} + (\rho - \ol{\lambda})$,
\[
  \pi^*c \in \ol{\lambda}\left( \, (\rho - \ol{\lambda})\left(\Gamma(U, \mathscr{O}_C(-\rho))\right) \, \right) \subseteq \ol{\lambda}(\Gamma(U, \mathscr{O}_C(-\ol{\lambda}))),
\]
so $\pi^*c$ vanishes on $E$. Therefore
\[
  \phi^{-1}(I_{\tau(E)}) = \Big\{ (f, c) \, \, \Big| \, \, f \in \Gamma(U, \mathscr{O}_U(-\ol{\lambda})) \text{ and }c \in \rho(\Gamma(S, \mathscr{O}_S(-\rho))) \Big\}
\]

Now, for any $(f,c) \in \phi^{-1}(I_{\tau(E)})$, we may find an element $d \in \Gamma(S, \mathscr{O}_S(-\rho))$ so that $c = -\rho(d)$. Using the $\ol{B}$-sequence, we may rewrite $\phi(f,c)$
as $\phi(f - (\rho - \ol{\lambda})(d), 0)$. This shows
\[
  I_{\tau(E)} = \Big\{ \phi(f, 0) \, \, \Big| \, \, f \in \Gamma(U, \mathscr{O}_U(-\ol{\lambda})) \Big\}.
\]
\end{proof}

\begin{proposition} \label{prop:iso_outside_E}
The map $\tau: U \to \ol{U}$ restricts to an isomorphism $U - E \to \ol{U} - \tau(E)$. In particular, the pushout defining $\ol{C}$ exists and the map $\tau : C - E \to \ol{C} - \tau(E)$ is also an isomorphism.
\end{proposition}
\begin{proof}
Since $\pi : E \to S$ is proper and $\ol{U} \to S$ is affine (hence separated), $\tau : E \to \ol{U}$ is proper. In particular, $\tau(E)$ is closed, so $\tau(E)$ coincides topologically with the scheme theoretic image of $E$.

As above, let $I_{\tau(E)} \subseteq \ol{B}(U)$ be the ideal of the scheme theoretic image of $E$ under $\tau$, and let $\phi : B(U) \to \ol{B}(U)$ be the quotient map.
Let $\phi(f,0) \in I_{\tau(E)}$. Since $\phi(f, 0)$ maps to $\ol{\lambda}(f)$ in $\Gamma(U, \mathscr{O}_C)$, we have a cartesian square
\begin{equation} \label{eq:tau_E_complement_affine_chart}
  \xymatrix{
    D(\ol{\lambda}(f)) \ar[d] \ar[r] & D( \phi(f,0) ) \ar[d] \\
    U \ar[r]^\tau & \ol{U}.
  }
\end{equation}
As $f$ varies over the elements of $\Gamma(U, \mathscr{O}_C(-\ol{\lambda}))$, the open subsets $D( \phi(f, 0) )$ vary over an affine open cover of $\ol{U} - \tau(E)$, and the $D(\ol{\lambda}(f))$ vary over an open cover of $U - E$.
Therefore it is enough to show that the top arrow is an isomorphism for each element $\phi(f, 0)$. Note that $D(\ol{\lambda}(f))$ is an open subset of $U$, which is not affine, so we do not yet know that $D(\ol{\lambda}(f))$ is affine.

We start by showing that the top row of (\ref{eq:tau_E_complement_affine_chart}) induces an isomorphism on global regular functions.
Note that $\Gamma(S, \mathscr{O}_S(-\rho)) \otimes_{B(U)} B(U)[(f,0)^{-1}] = 0$, as $(f,0)$ acts as 0 on $\Gamma(S, \mathscr{O}_S(-\rho))$. Then, tensoring the $\ol{B}$-sequence
\[
  0 \to \Gamma(S, \mathscr{O}_S(-\rho)) \to B(U) \overset{\phi}{\to} \ol{B}(U) \to 0
\]
with the flat $B(U)$-module $B(U)[(f,0)^{-1}]$, we get
\[
  0 \to B(U)[(f,0)^{-1}] \overset{\sim}{\to} \mathscr{O}_{\ol{U}}(D(\phi(f,0))) \to 0.
\]

Now consider the map
\[
  B(U)[(f,0)^{-1}] \to \Gamma(U, \mathscr{O}_C)[\ol{\lambda}(f)^{-1}] \cong \Gamma(D(\ol{\lambda}(f)), \mathscr{O}_C)
\]
induced by tensoring $B(U) \to \Gamma(U, \mathscr{O}_C)$ with $B(U)[(f,0)^{-1}]$.
We'd like to show that this map is an isomorphism. For injectivity, suppose $(g,c) \cdot (f,0)^{-n}$ maps to 0. Then $\ol{\lambda}(f)^k (\ol{\lambda}(g) + c) = 0$ for some $k$.
Now, multiplying by a unit,
\begin{align*}
  (g,c) \cdot (f,0)^{-n} \cdot (f,0)^{n + k + 1} &= (g, c) \cdot ( (\ol{\lambda}(f))^k f, 0) \\
    &= \Big((\ol{\lambda}(g) + c) \cdot \ol{\lambda}(f)^k f, \, 0\Big) \\
    &= 0
\end{align*}
so $(g,c) \cdot (f,0)^{-n}$ must have been zero to begin with in $B(U)[(f,0)^{-1}]$.

For surjectivity, suppose $g \cdot (\ol{\lambda}(f))^{-n} \in \Gamma(U, \mathscr{O}_C)[\ol{\lambda}(f)^{-1}]$. This has preimage $(f \cdot g, 0) \cdot (f, 0)^{-(n + 1)}$ since
\begin{align*}
  (f \cdot g, 0) \cdot (f, 0)^{-(n + 1)} &\mapsto \ol{\lambda}(f \cdot g) \ol{\lambda}(f)^{-(n + 1)} \\
   &= g \cdot \ol{\lambda}(f)^{-n}.
\end{align*}

The map on global regular functions induced by $\tau$ is the composite isomorphism from $\mathscr{O}_{\ol{U}}(D(\phi(f,0)))$ to $\mathscr{O}_C(D(\ol{\lambda}(f)))$. The same argument shows that if $T \to S$ is an \'etale cover,
\[
  D(\ol{\lambda}(f)) \times_S T \cong D(\ol{\lambda}(f|_{C \times_S T})) \longrightarrow D(\phi(f|_{C \times_S T}, 0)) \cong D(\phi(f, 0)) \times_S T
\]
also induces an isomorphism on global regular functions.

To complete the proof, we check that $D(\ol{\lambda}(f))$ is affine \'etale locally on $S$. Let $\Sigma$ be the sum of the divisors of the sections $\sigma_i$. Replacing $S$ by an \'etale cover if necessary, we can find a finite union of sections
$\Sigma'$, passing through the smooth locus of each component of each fiber of $E$ and disjoint from each other and the $\sigma_i$s. Consider the sheaves $\mathscr{O}_C(n(\Sigma + \Sigma'))$. Arguing as in Proposition \ref{prop:no_coh_minus_lambda_bar}, we see that for a sufficiently large
$n$, $H^1(C, \mathscr{O}_C(n(\Sigma + \Sigma'))) = 0$, since this holds on fibers. Then we can check that $\mathscr{O}_C(n(\Sigma + \Sigma'))$ is relatively very ample on fibers; this is clear for sufficiently large $n$. Consider the associated
embedding of $C$ into projective space over $S$; the locus $U' = C - (\Sigma \cup \Sigma')$ is a closed subset of the complement of some hyperplane, so affine. Then $D(\ol{\lambda}(f)) = U - V(\ol{\lambda}(f)) = U' - V(\ol{\lambda}(f)|_{U'})$ is a
distinguished open subset of $U'$, so it is affine too.
\end{proof}

\begin{corollary}
The map $\tau : C \to \ol{C}$ is surjective.
\end{corollary}
\begin{proof}
We have $\tau$ sends $U - E$ onto $U - \tau(E)$. We noted in the course of the previous proof that the image of $E$ is $\tau(E)$. Gluing with the rest of $C$, $\tau$ is surjective.
\end{proof}

Next, we must verify is that $\ol{C}$ does not depend on our choice of sections $\sigma_i$.

\begin{proposition} \label{prop:section_invariant}
Suppose we are in the standard situation and $\sigma'_1, \ldots, \sigma'_k$ are a second choice of sections. There is a uniquely
specified isomorphism $\phi_{\sigma_\bullet, \sigma_\bullet'} : \ol{(C, \sigma_\bullet)} \to \ol{(C, \sigma_\bullet')}$ so that
\[
  \xymatrix{
    C_T \ar[rdd] \ar[rr] \ar[rd]& & \ol{(C, \sigma_\bullet')}  \ar[ldd] \\
     & \ol{(C, \sigma_\bullet)} \ar[d] \ar^{\phi_{\sigma_\bullet, \sigma_\bullet'}}[ru]\\
    & T
  }
\]
commutes. Moreover, if $\sigma''_1, \ldots, \sigma''_l$ is a third choice of sections, the cocycle condition
\[
  \phi_{\sigma_\bullet, \sigma_\bullet''} = \phi_{\sigma_\bullet'', \sigma_\bullet'} \circ \phi_{\sigma_\bullet, \sigma_\bullet'}
\]
holds.
\end{proposition}
\begin{proof}
Let $U'$ be the complement of the sections $\sigma_1', \ldots, \sigma_k'$. Recall that for any open subset $V$ of $C$ containing $E$ we have $V \to \ol{V} = \Spec \ol{B}(V)$, as in the discussion preceding Definition \ref{def:contraction}.

Our technique will be to show that
\begin{equation} \label{eq:uprime_u_cartesian_square}
   \xymatrix{
       U' \cap U \ar[d] \ar[r] & U \ar[d] \\
      \ol{U' \cap U} \ar[r] & \ol{U}
   },
\end{equation}
is cocartesian. Then the tall commutative square in
\[
  \xymatrix{
    (U' \cap U) - E \ar[r] \ar[d] & U - E \ar[d] \\
    U' \cap U \ar[d] \ar[r] & U \ar[d] \\
      \ol{U' \cap U} \ar[r] & \ol{U}
  }
\]
is cocartesian, implying that both the right square and wide square in
\[
  \xymatrix{
    (U' \cap U) - E \ar[r] \ar[d] & U - E \ar[r] \ar[d] & C - E \ar[d] \\
    \ol{U' \cap U} \ar[r] & \ol{U} \ar[r] & \ol{C}
  }
\]
are cocartesian. Symmetrically, both the right square and wide square of
\[
  \xymatrix{
    (U' \cap U) - E \ar[r] \ar[d] & U' - E \ar[r] \ar[d] & C - E \ar[d] \\
    \ol{U' \cap U} \ar[r] & \ol{U'} \ar[r] & \ol{C}
  }
\]
are cocartesian. Then $(C, \sigma_\bullet)$ and $(C, \sigma_\bullet')$ share the universal property of the pushout for the wide square, yielding the unique isomorphism $\phi_{\sigma_\bullet, \sigma_\bullet'}$. A similar argument for a triple
intersection gives us the cocycle condition.

Now, in order to show that the square (\ref{eq:uprime_u_cartesian_square}) is cocartesian, let $Z$ be the complement of $U' \cap U$ in $U$. Since $Z$ is a closed subset of $U - E$ and $U - E \to \ol{U} - \tau(E)$ is an isomorphism,
$\tau(Z)$ is a closed subset of $\ol{U}$ not meeting $\tau(E)$. Giving $\tau(Z)$ the reduced scheme structure, $\tau(Z)$ has an ideal $I_{\tau(Z)}$ in $\ol{B}(U)$.

Recall that $\phi : B(U) \to \ol{B}(U)$ is the quotient map. We claim that Zariski locally on $S$ there is an element of $\phi^{-1}(I_{\tau(Z)})$ of the form $(f, c)$ where $c$ is a unit. If not, then we may choose a point $s$ of $S$ where
$(f, c) \equiv (f, 0) \mod k(s)$ for all $(f,c) \in \phi^{-1}(I_{\tau(Z)})$. This implies that $\tau(E)$ has non-trivial intersection with $\tau(Z)$ in the fiber over $s$, a contradiction.

Since the square (\ref{eq:uprime_u_cartesian_square}) is a diagram of $S$-schemes, we may check that it is a pushout square Zariski locally on $S$. Passing to an open cover of $S$ if necessary, we may assume that there is an element $(f,c) \in \phi^{-1}(I_{\tau(Z)})$ with $c$ a unit. Let $V$ be the complement of the vanishing of $\ol{\lambda}(f) + c$ in $U$. Consider the diagram
\[
  \xymatrix{
       V \ar[r] \ar[d] & U' \cap U \ar[d] \ar[r] & U \ar[d] \\
       \ol{V} \ar[r] & \ol{U' \cap U} \ar[r] & \ol{U}
   }
\]
A standard diagram chase shows that if the left square and wide square are cocartesian, then the right square is cocartesian. Note that $U' \cap U$ with sections $\{ \sigma_i \} \cup \{ \sigma_j' \}$ is also in the standard situation.  The argument that shows that the left square is cocartesian is identical to the argument that the wide square is cocartesian,
so we just give the argument for the latter.

Now,
\[
  \xymatrix{
     V \ar[r] \ar[d] & U \ar[d] \\
     \ol{U} - V(\phi(f,c)) \ar[r] & \ol{U}
  }
\]
is trivially cocartesian, since $\tau : U - E \to \ol{U} - \tau(E)$ is an isomorphism and the vanishing locus of $\phi(f,c)$ does not meet $\tau(E)$. It suffices, therefore, to show that there is an isomorphism $\ol{V} \to \ol{U} - V(\phi(f,c))$ commuting with the maps from $V$.
One checks via a direct computation that the map $\ol{B}(U)[\phi(f,c)^{-1}] \to \ol{B}(V)$ induced by
\[
  (g,d) \cdot (f,c)^{-n} \mapsto (g,d) \cdot \left( \frac{-f}{c(\ol{\lambda}(f) + c)}, c^{-1}\right)^n
\]
gives the required isomorphism.
\end{proof}

\subsection{Regular functions near the singular point} \label{ssec:reg_functions}

Our goal in this section will be to describe the regular functions in a neighborhood of $\tau(E)$ --- that is, the ring $\ol{B}(U)$ --- when the base is
the spectrum of an algebraically closed field. This provides us with a description of the singularities that can arise in $\ol{C}$, and we give several examples to illustrate some of the possibilities. In particular, we will see that $\tau$ contracts $E$ to a singularity with the same genus, as required for $\tau$ to be a contraction. We will also see that $\ol{\pi} : \ol{C} \to S$ has reduced geometric fibers, as required for $\ol{\pi}$ to be a family of curves. 

We start with a description of $\Gamma(U, \mathscr{O}_C(-\ol{\lambda}))$.

\begin{proposition} \label{prop:sections_of_oc_minus_lambda_bar_on_fiber}
Suppose $S = \Spec k$ is a geometric point and $E$ is nonempty with arithmetic genus $g$.
Let $Z$ be the union of the irreducible components of $U$ not contained in $E$, and let $p_1, \ldots, p_m$ be the points in which $E$ meets $Z$.
Finally, let $V$ be the $k$-vector space of tuples $\{ (a_1, \ldots, a_m) \} \subseteq k^m$ so that there exists $\sigma \in \Gamma(E, \mathscr{O}_E(-\ol{\lambda}))$
 with $\sigma(p_i) = a_i$ for all $i = 1, \ldots, m$.

Then there is an exact sequence
\[
  0 \to \Gamma(S, \mathscr{O}_S(-\rho)) \overset{\rho - \ol{\lambda}}{\longrightarrow} \Gamma(U, \mathscr{O}_C(-\ol{\lambda})) \to \Gamma(Z, \mathscr{O}_Z(-\ol{\lambda})) \to k^m/V \to 0,
\]
where the second map is induced by restriction from $U$ to $Z$, and the third map is induced by evaluation at the points $p_i$.
Moreover, the first map admits a splitting, so that
\[
  \Gamma(U, \mathscr{O}_C(-\ol{\lambda})) \cong \Gamma(S, \mathscr{O}_S(-\rho)) \oplus \bigg\{ f \in \Gamma(Z, \mathscr{O}_Z(-\ol{\lambda})) \, \, \bigg| \, \, [f(p_i)]_{i = 1}^n \in V \bigg\}.
\]
\end{proposition}

\begin{proof}
Let $F$ be the top of the mesa $\ol{\lambda}$. Note that $\rho = \ol{\lambda}$ on the smooth locus of $F$ and $\rho > \ol{\lambda}(v)$ for the components of $\ol{C - F}$.
It follows that the map $\pi^{-1}\mathscr{O}_S(-\rho) \to \mathscr{O}_U(-\rho) \overset{\rho - \ol{\lambda}}{\longrightarrow} \mathscr{O}_U(-\ol{\lambda})$ is induced by multiplication by a unit on the smooth locus of $F$ and multiplication
by zero elsewhere.

Now, we have a short exact sequence
\[
  0 \to \mathscr{O}_U(-\ol{\lambda}) \to \mathscr{O}_E(-\ol{\lambda}) \oplus \mathscr{O}_Z(-\ol{\lambda}) \to \bigoplus_{i = 1}^m k(p_i) \to 0.
\]
This gives us an exact sequence
\[
  0 \to \Gamma(U, \mathscr{O}_U(-\ol{\lambda})) \to \Gamma(E, \mathscr{O}_E(-\ol{\lambda})) \oplus \Gamma(Z, \mathscr{O}_Z(-\ol{\lambda})) \to k^m.
\]
That is, a section of $\mathscr{O}_U(-\ol{\lambda})$ is the same as a pair of sections of $\mathscr{O}_E(-\ol{\lambda})$ and $\mathscr{O}_Z(-\ol{\lambda})$ that agree at the points $p_i$.

Let $f_Z \in \Gamma(Z, \mathscr{O}_Z(-\ol{\lambda}))$, and consider the problem of extending $f_Z$ to $\Gamma(U, \mathscr{O}_C(-\ol{\lambda}))$.
This is equivalent to finding a section $f_E \in \Gamma(E, \mathscr{O}_E(-\ol{\lambda}))$ so that $f_E(p_i) = f_Z(p_i)$ for all $i$.
By definition, such an $f_E$ exists if and only if $[ f_Z(p_i) ]_{i = 1}^m$ lies in $V$.
By Proposition \ref{prop:generalcurvesectionvalues}, for each $f_Z$ satisfying this condition, there is precisely a 1-dimensional family of such $f_E$s. In fact, since the image $I$ of
$k \cong \Gamma(S, \mathscr{O}_S(-\rho)) \overset{\rho - \ol{\lambda}}{\longrightarrow} \mathscr{O}_E(-\ol{\lambda})$ is 1-dimensional and
any $h \in I$ satisfies $h(p_i) = 0$ for all $i$, every choice of $f_E$ in the 1-dimensional family is obtained by adding an element of $I$. That is,
\[
  0 \to \Gamma(S, \mathscr{O}_S(-\rho)) \overset{\rho - \ol{\lambda}}{\longrightarrow} \Gamma(U, \mathscr{O}_C(-\ol{\lambda})) \to \Gamma(Z, \mathscr{O}_Z(-\ol{\lambda})) \to k^m/V \to 0
\]
is exact.

To show that we have a splitting of the first map, let $\sigma : S \to C$ be any section through the smooth locus of $F$. Then $f \mapsto \sigma^*f$ gives a splitting of $\Gamma(S, \mathscr{O}_S(-\rho)) \to \Gamma(U, \mathscr{O}_C(-\ol{\lambda}))$,
since $\rho - \ol{\lambda} = 0$ on the smooth locus of $F$.
\end{proof}

We will want to identify when the singularity at $\tau(E)$ is elliptic Gorenstein, so we recall a convenient characterization from \cite{smyth_mstable}.

\begin{lemma} \label{lem:elliptic_m_fold_characterization} \cite[Lemma 2.2]{smyth_mstable}
Let $\pi : C \to \Spec k$ be a reduced curve and let $\nu : \tilde{C} \to C$ be the normalization. Let $p \in C$ be a point with pre-images $p_1, \ldots, p_m$ in $\tilde C$. Then $p$ is an elliptic Gorenstein singularity if and only if $\nu^* : \hat{\mathscr{O}}_{C, p} \to \hat{\mathscr{O}}_{\tilde C, \nu^{-1}(p)}$ satisfies
\begin{enumerate}
  \item $\nu^*(m_p/m_p^2) \subset \bigoplus_{i = 1}^m m_{p_i}/m_{p_i}^2$ is a codimension-one subspace.
  \item $\nu^*(m_p/m_p^2) \not\supseteq m_{p_i}/m_{p_i}^2$ for any $i = 1, \ldots, m$.
  \item $\nu^*(m_p) \supseteq \bigoplus_{i = 1}^m m_{p_i}^2$.
\end{enumerate}
\end{lemma}

\begin{proposition} \label{prop:singularity_has_correct_genus}
Suppose $S = \Spec k$ is a geometric point and $E$ is nonempty with arithmetic genus $g$.
Let $Z$ be the union of the irreducible components of $U$ not contained in $E$, and let $p_1, \ldots, p_m$ be the points in which $E$ meets $Z$.
Finally, let $V$ be the $k$-vector space of tuples $\{ (a_1, \ldots, a_m) \} \subseteq k^m$ so that there exists $\sigma \in \Gamma(E, \mathscr{O}_E(-\ol{\lambda}))$
 with $\sigma(p_i) = a_i$ for all $i = 1, \ldots, m$.

Then
\begin{enumerate}
  \item $V$ is a subspace of $k^m$ of codimension $g$;
  \item We may identify $\ol{B}(U)$ with the ring
    \[
      \bigg\{ (f, c) \in \Gamma(Z, \mathscr{O}_Z(-\ol{\lambda})) \oplus k \, \, \bigg| \, \, [ f(p_i) ]_{i = 1}^m \in V \bigg\}
    \]
    where the multiplication is given by $(f, c) \cdot (f', c') = (\ol{\lambda}(ff') + fc' + f'c, cc')$;
  \item With respect to this presentation, the ideal $I_{\tau(E)} \subseteq \ol{B}(U)$ is the subset of $\ol{B}(U)$ where $c = 0$;
  \item $\tau(E)$ is a singularity of genus $g$;
  \item If $\ol{\lambda}$ is steep and $E$ has genus one, then $\tau(E)$ is elliptic Gorenstein with $m$ branches.
\end{enumerate}
In particular, for an arbitrary mesa curve $\pi : C \to S$, the geometric fibers of $\ol{\pi}: \ol{C} \to S$ are reduced and one-dimensional. 
\end{proposition}
\begin{proof}
Part (i) is immediate from Proposition \ref{prop:generalcurvesectionvalues}.

Consider the truncation
\[
  0 \to \Gamma(S, \mathscr{O}_S(-\rho)) \overset{\rho - \ol{\lambda}}{\longrightarrow} \Gamma(U, \mathscr{O}_C(-\ol{\lambda}))  \to \Big\{ f \in \Gamma(Z, \mathscr{O}_Z(-\ol{\lambda})) \, \, \Big| \, \, [f(p_i)]_{i = 1}^n \in V \Big\} \to 0
\]
of the long exact sequence of the previous proposition. Note that $\rho : \Gamma(S, \mathscr{O}_S(-\rho)) \to \Gamma(S, \mathscr{O}_S)$ is the zero map, since $E$ is nonempty and $S$ is a geometric point. 
Tacking on summands of $\Gamma(S, \mathscr{O}_S) \cong k$ to the last two terms, it follows
\[
  0 \to \Gamma(S, \mathscr{O}_S(-\rho)) \overset{\rho - \ol{\lambda} \oplus -\rho}{\longrightarrow} B(U) \to \Big\{ (f, c) \in \Gamma(Z, \mathscr{O}_Z(-\ol{\lambda})) \oplus k \, \, \Big| \, \, [ f(p_i) ]_{i = 1}^m \in V \Big\} \to 0
\]
is exact.
Then, by definition of $\ol{B}(U)$, we may identify $\ol{B}(U)$ with the ring
\[
  \Big\{ (f, c) \in \Gamma(U, \mathscr{O}_Z(-\ol{\lambda})) \oplus k \, \, \Big| \, \, [ f(p_i) ]_{i = 1}^m \in V \Big\}
\]
with multiplication $(f, c) \cdot (f', c') = (\ol{\lambda}(ff') + fc' + f'c, cc')$. This proves (ii).

(iii) is clear from the expression for $I_{\tau(E)}$ in Lemma \ref{lem:expression_for_I_tau_E}.

Now, consider the factorization
\[
  \xymatrix{
    U \ar[r] \ar[rd]_\tau & \Spec \Gamma(U, \mathscr{O}_U) \ar[d]^{\ol{\tau}} \\
   & \ol{U}
  }
\]
of $\tau$.
Note that
\begin{align*}
  \Gamma(U, \mathscr{O}_U) &\cong \{ f \in \Gamma(Z,\mathscr{O}_Z) \mid f(p_i) = f(p_j) \text{ for all }i,j \} \\
     &\cong \{ (f, c) \in \Gamma(Z, \mathscr{O}_Z) \oplus k \mid f(p_i) = 0 \text{ for all } i \} \\
     &\cong \left\{ (f ,c) \in \Gamma\left(Z, \mathscr{O}_Z\left(-\sum\nolimits_i p_i\right)\right) \oplus k \right\} \\
     &\cong \{ (f, c) \in \Gamma(Z, \mathscr{O}_Z(-\ol{\lambda})) \oplus k \},
\end{align*}
with multiplication $(f, c) \cdot (f', c') = (\ol{\lambda}(ff') + fc' + f'c, cc')$. Written this way, we see that $\ol{\tau}$ identifies $\ol{B}(U)$ with the subring of $\Gamma(U, \mathscr{O}_U)$ cut out by the condition $[ f(p_i) ]_{i = 1}^m \in V$.

It is clear that $Z$ is the normalization of both $\Spec \Gamma(U, \mathscr{O}_U)$ and $\ol{U}$. Since $\ol{B}(U)$ has codimension $g$ in $\Gamma(U, \mathscr{O}_U)$, $\tau(E)$ has genus $g$. This shows (iv).

In order to show (v), suppose that $\ol{\lambda}$ is steep and $E$ has genus one. Proposition \ref{prop:generalcurvesectionvalues} tells us that there are nonzero constants $c_1, \ldots, c_m$ so that $[ f(p_i) ]_{i = 1}^m \in V \iff \sum_{i = 1}^m c_if(p_i) = 0$. One verifies easily that the conditions of Lemma \ref{lem:elliptic_m_fold_characterization} hold: (i) holds since we
impose the condition $\sum_{i = 1}^m c_if(p_i) = 0$, (ii) holds since each $c_i$ is nonzero, and (iii) holds since the condition $\sum_{i = 1}^m c_if(p_i) = 0$ holds if each $f(p_i)$ is zero.
\end{proof}

\begin{example}
Let $C$ be a nodal curve over $S = \Spec k$, $k$ algebraically closed, so that $C$ consists of a smooth genus $g$ component, $E$, and $g$ rational curves $Z_1, \ldots, Z_g$ attached to $E$ at distinct points
$p_1, \ldots, p_g$, so that $H^1(E, \mathscr{O}_E(p_1 + \cdots + p_g)) = 0$. (A generic choice of $p_1, \ldots, p_g \in E$ will do the job.)

We can give $S$ the log structure associated to the chart $\N\delta \to k$ sending $\delta \mapsto 0$. Then $C$ may be given a compatible log structure where $\delta_{p_i}$ is $\delta$ for each $i$. Then
\[
  \ol{\lambda}_v = \begin{cases}
     \delta & \text{ if } v = E \\
     0 & \text{ otherwise }
  \end{cases}
\]
defines a piecewise linear function on $\Gamma(C)$ so that $(C, \ol{\lambda})$ is a mesa curve. Choose the open set $U$ as the complement of the points at infinity on the $Z_i$s, and let $Z = U \cap \bigcup_{i = 1}^g Z_i$.
By Proposition \ref{prop:singularity_has_correct_genus},
\begin{align*}
  \ol{B}(U) \cong \bigg\{ f \in \Gamma(Z, \mathscr{O}_Z(-\ol{\lambda})) \oplus k \, \, \bigg| \, \,  [f(p_i)]_{i = 1}^g = \vec{0} \bigg\},
\end{align*}
since $V$ is a codimension $g$ subspace of $k^g$. Simplifying further,
\begin{align*}
  \ol{B}(U) \cong &\bigg\{ f \in \Gamma(Z, \mathscr{O}_Z(-p_1 - \cdots - p_g)) \oplus k \, \, \bigg| \, \,  f(p_i) = 0 \text{ for all }i \bigg\} \\
   \cong &\bigg\{ f \in \Gamma(Z, \mathscr{O}_Z(-2p_1 - \cdots - 2p_g)) \oplus k  \bigg\} \\
   \cong &\bigg\{ (f_1, \ldots, f_g) \in k[x_1^2, x_1^3] \times \cdots \times k[x_g^2, x_g^3] \, \, \bigg| \, \, f_1(0) = \cdots = f_g(0) \bigg\},
\end{align*}
where $x_1, \ldots, x_g$ are local parameters of $p_1, \ldots, p_g$ on $Z_1, \ldots, Z_g$, respectively. That is, the singularity at $\tau(E)$ consists of $g$ cusps glued transversally at their singular points.
\end{example}

\begin{example}

When a genus one mesa is not steep, the resulting singularity need not be Gorenstein. Consider a log curve over $S = \Spec k$, $k$ algebraically closed, with tropicalization
\begin{center}
\includegraphics{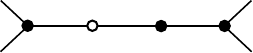}
\end{center}
where the filled dots have genus zero and the empty dot has genus one. From left to right, name the components $Z_1, E_1, E_2$, and $Z_2$.
Assume that the edges connecting $Z_1$ with $E_1$ and $Z_2$ with $E_2$ have the same length, $\delta$. Then we can give $C$ a mesa
$\ol{\lambda}$ with values
\[
  \ol{\lambda}(Z_1) = 0, \quad \ol{\lambda}(E_1) = \delta, \quad \ol{\lambda}(E_2) = \delta, \quad \ol{\lambda}(Z_2) = 0.
\]
Let $p_1$ be the point where $Z_1$ meets $E_1$, and let $p_2$ be the point where $Z_2$ meets $E_2$.
Choose $U$ as the complement of the points at infinity of $Z_1$ and $Z_2$, and let $Z = U \cap (Z_1 \cup Z_2)$.

By Proposition \ref{prop:singularity_has_correct_genus}, we know
\[
  \ol{B}(U) \cong \bigg\{ f \in \Gamma(Z, \mathscr{O}_Z(-\ol{\lambda})) \oplus k \, \, \bigg| \, \, \begin{bmatrix} f(p_1) \\ f(p_2) \end{bmatrix} \in V \bigg\}
\]
where $V$ is some codimension 1 subspace of $k^2$.

Now, $V$ is the set of pairs $(a_1, a_2) \in k^2$ so that there exists a global section
$\sigma \in \Gamma(E, \mathscr{O}_E(-\ol{\lambda}))$ with $\sigma(p_1) = a_1$ and $\sigma(p_2) = a_2$.
As a start, we may ask if there exists $\sigma_1 \in \Gamma(E_1, \mathscr{O}_{E_1}(-\ol{\lambda})) \cong \Gamma(E_1, \mathscr{O}_{E_1}(p_1))$ with $\sigma_1(p_1) = a_1$. The exact sequence
\[
  0 \to \mathscr{O}_{E_1} \to \mathscr{O}_{E_1}(p_1) \to k(p_1) \to 0
\]
gives rise to the exact sequence
\[
  0 \to k \to \Gamma(E_1, \mathscr{O}_{E_1}(p_1)) \to k \to k \to 0,
\]
which tells us that every $\sigma_1 \in \Gamma(E_1, \mathscr{O}_{E_1}(p_1))$
satisfies $\sigma_1(p_1) = 0$. This same condition must hold for any $\sigma \in \Gamma(E, \mathscr{O}_E(-\ol{\lambda}))$, so the codimension 1 condition cutting out $V$ is $a_1 = 0$.

Now, simplifying,
\begin{align*}
  \ol{B}(U) \cong &\bigg\{ f \in \Gamma(Z, \mathscr{O}_Z(-\ol{\lambda})) \oplus k \, \, \bigg| \, \, \begin{bmatrix} f(p_1) \\ f(p_2) \end{bmatrix} \in V \bigg\} \\
    \cong &\bigg\{ f \in \Gamma(Z, \mathscr{O}_Z(-p_1 - p_2)) \oplus k \, \, \bigg| \, \, f(p_1) = 0 \bigg\} \\
    \cong &\bigg\{ (f, g) \in k[x^2, x^3] \times k[y] \, \, \bigg| \, \,  f(0) = g(0) \bigg\},
\end{align*}
where $x$ and $y$ are local parameters of $p_1$ and $p_2$ on $Z_1$ and $Z_2$, respectively. That is, $\ol{C}$ consists of a rational curve with a cusp glued transversally to a smooth rational curve.

More generally, if $E_1$ were attached to $k$ rational curves and $E_2$ were attached to $l$ rational curves, $\tau(E)$ would consist of an elliptic $k$-fold Gorenstein singularity glued transversally
to a rational singularity with $l$ branches.
\end{example}

\begin{example} \label{ex:tacnode} (Variation of log structure in a contraction to a tacnode)

Suppose $E$ is a smooth elliptic curve over an algebraically closed field $k$, and $p, q$ are two closed points of $E$. Form a curve $C_0$
by attaching rational components $Z_p$ and $Z_q$ to $E$ at $p$ and $q$, respectively. If $C_0$ is given the structure of a mesa
curve whose mesa has support $E$, then we expect $\ol{C_0}$ to consist of two rational components meeting in an elliptic Gorenstein
singularity with 2 branches: a tacnode.

Consider the Stein factorization of the contraction map,
\[
  \xymatrix{
    C_0 \ar[r] \ar[rd]_\tau & C_0^+ = (\ol{C_0}, \tau_*\mathscr{O}_{C_0}) \ar[d]^{\ol{\tau}} \\
     & \ol{C_0}.
  }
\]

The intermediate curve $C_0^+$ consists of two rational components meeting in an ordinary double point, and the map
$\ol{\tau}$
collapses the ordinary double point to the tacnode. There is a one parameter family of such maps. To see this, choose a local parameter $y_p$
of $p$ in $Z_p$ and a local parameter $y_q$ of $q$ in $Z_q$. Near the node, $C_0^+$ has functions
\[
  \{ (f(y_p), g(y_q)) \mid f(0) = g(0) \}
\]
and $\ol{\tau}$ could be the inclusion of any of the subrings defined by $a\frac{df}{dy_p}(0) = \frac{dg}{dy_q}(0)$ for some choice of $a \in k^*$.

Each possible contraction is realized for a different choice of log structure on $C_0$. Let us see how.

Let $S = \Spec k[z, z^{-1}]$, $C = C_0 \times S$, and let $\pi : C \to S$ be the projection onto the second factor.
Give $S$ the log structure associated to the chart $\N \delta \to k[z, z^{-1}]$ where $\delta \mapsto 0$.
Choose an \'etale neighborhood
$U_p$ of $p$ in $C_0$ so that there is an \'etale map $U_p \to \Spec k[x_p, y_p]/(x_py_p)$. We can pull this back to an \'etale neighborhood
of $p \times S$:
\[
  \xymatrix{
    C \ar[d] & \ar[l]_{\text{\'et}} U_p \times S \ar[d]^{\text{\'et}} \\
      S  & \Spec k[z, z^{-1}, x_p, y_p]/(x_py_p) \ar[l]
  }
\]
Give $\Spec k[z, z^{-1}, x_p, y_p]/(x_py_p)$ the log structure associated to the chart
\begin{align*}
  \N\alpha_p \oplus \N\beta_p \oplus \N\delta / (\alpha_p + \beta_p \sim \delta) &\longrightarrow k[z, z^{-1}, x_p, y_p]/(x_p, y_p) \\
   \alpha_p \quad &\mapsto \quad x_p \\
   \beta_p \quad &\mapsto \quad y_p \\
   \delta \quad & \mapsto \quad 0,
\end{align*}
then give $U_p \times S$ the pulled back log structure. (We are taking advantage of the isomorphism
\begin{align*}
  \N\alpha \oplus \N\beta \oplus P / (\alpha + \beta \sim \delta) &\longrightarrow \{ (p,q) \in P \oplus P \mid p - q \in \Z\delta \} \\
   (n, m, p) &\mapsto (p + n\delta, p + m\delta)
\end{align*}
to specify the chart.)

We may repeat the process for $q$. This time, we give $U_q \times S$ the log structure pulled back from the chart (note the image of $\beta_q$!)
\begin{align*}
  \N\alpha_q \oplus \N\beta_q \oplus \N\delta / (\alpha_q + \beta_q \sim \delta) &\longrightarrow k[z, z^{-1}, x_q, y_q]/(x_q, y_q) \\
   \alpha_q \quad &\mapsto \quad x_q \\
   \beta_q \quad &\mapsto \quad zy_q \\
   \delta \quad & \mapsto \quad 0.
\end{align*}
Give $C - (\{ p, q \} \times S)$ the log structure pulled back from $S$. These local descriptions glue together in the evident way to
give $C$ a log structure. Letting $E$ be the mesa, $C \to S$ gains the structure of a mesa curve.

For $a \in k^*$, let $C_a$ be the fiber of $C$ over $z = a$. $C_a$ is isomorphic to $C_0$ as a
scheme in a canonical way. The piecewise linear function associated to $\ol{\lambda}$ is constant with respect to $a$, and the sheaf $\mathscr{O}_{C_a}(-\ol{\lambda})$ has the same isomorphism class in all fibers.
What varies is the map $\mathscr{O}_{C_a}(-\ol{\lambda}) \to \mathscr{O}_{C_a}$.

To see this, we write the sections of $\mathscr{O}_{C_a}(-\ol{\lambda})$ with respect to the the cover $U_p$, $U_q$, $E_a -  \{ p, q \}$,
$C_a - E_a$
on which $\ol{\lambda}$ lifts to $\beta_p$, $\beta_q$, $\delta$, and $0$ respectively. Let $U$ be the complement of suitably chosen points of
$C_a$, as usual. Note that the chosen lifts of $\ol{\lambda}$ all coincide on overlaps, so $\Gamma(U, \mathscr{O}_{C_a}(-\ol{\lambda}))$ is equal to
\[
  \Big\{ (h_p, h_q, h_E, h_{C-E}) \in (\mathscr{O}_{U_p} \times \mathscr{O}_{U_q} \times \mathscr{O}_{E_a - \{p, q \} } \times \mathscr{O}_{C_a - E_a})(U) \, \, \Big| \, \, h_i = h_j \text{ for all }i, j\Big\},
\]
where the equation on the right is to be interpreted as equality of each pair of functions on the intersection of their domains.

Since there is no twisting, $h_p, h_q, h_E$ are constant on $E_a$. In particular, $f_p(p) = f_q(q)$. The map to $\Gamma(U, \mathscr{O}_C)$ is
induced by multiplying $h_p, h_q, h_E, h_{C - E}$ by $y_p$, $zy_q$, $0$, and $1$, respectively. Then $\mathscr{O}_{\ol{C_a}}(U)$ is the subring of
\[
  \{ (f(y_p), g(y_q)) \mid f(0) = g(0) \}
\]
determined by $a\frac{df}{dy_p}(0) = a h_p(p) = a h_q(q) = \frac{dg}{dy_q}(0)$.
\end{example}

\subsection{Properness}

In this section we complete the proof that $\ol{\pi} : \ol{C} \to S$ is a family of curves. We already know that it is flat and has reduced geometric
fibers. It remains to show that $\ol{\pi} : \ol{C} \to S$ is proper. We do this in two steps: first we show that $\ol{\pi} : \ol{C} \to S$ is separated.
It follows that $\tau : C \to \ol{C}$ is proper. Then we show that $\tau$ has $\ol{C}$ as its scheme-theoretic image. Since the image of a proper
map in a separated scheme is proper, we then have that $\ol{\pi} : \ol{C} \to S$ is proper as desired.

\begin{proposition}
The morphism $\ol{\pi}: \ol{C} \to S$ is separated.
\end{proposition}
\begin{proof}
Choose an affine open cover $\{ V_i \}_{i \in I}$ of $C - E$. We will identify the $V_i$s with their isomorphic images in $\tau(C - E) = \ol{C} - \tau(E)$. Then $\mathcal{U} = \{ \ol{U} \} \cup \{ V_i \}_{i \in I}$ is an affine open cover of $\ol{C}$.
It suffices to show that for each $V, W  \in \mathcal{U}$, $V \cap W \overset{\Delta}{\to} V \times_S W$ is a closed immersion. When $V = W = \ol{U}$, this is clear since affine maps are separated. If $V = V_i$ and $W = V_j$, we have that $V \cap W \overset{\Delta}{\to} V \times_S W$ is
a closed immersion, since $\pi : C \to S$ is separated.

It remains to check that each map $\ol{U} \cap V_i \overset{\Delta}{\to} \ol{U} \times_S V_i$ is a closed immersion for each $i$. Note that since $V_i$ is disjoint from $E$, $(U - E) \cap V_i = U \cap V_i$ and $\tau : U \cap V_i \to \ol{U} \cap V_i$ is an isomorphism.
Consider the diagram
\[
  \xymatrix{
     (U - E) \cap V_i \ar[r]^{\Delta} \ar@{=}[d] & (U - E) \times_S V_i \ar[d] \\
     U  \cap V_i \ar[r] \ar[d]^{\sim} \ar[r]^{\Delta} & U \times_S V_i \ar[d] \\
    \ol{U} \cap V_i \ar[r]^{\Delta} & \ol{U} \times_S V_i
  }
\]
It is now enough to show that the image of $(U - E) \cap V_i$ in $\ol{U} \times_S V_i$ is closed. We know that the image of $(U - E) \cap V_i \to (U - E) \times_S V_i$ is closed, since $\pi : C \to S$ is separated.
Since $(U - E) \times V_i \to \ol{U} \times_S V_i$ is an open immersion and $E \to \ol{U}$ surjects onto the complement of $U - E$ in $\ol{U}$, it suffices to show that the closure of the
image of $(U - E) \cap V_i \to \ol{U} \times_S V_i$ does not intersect $\tau(E) \times_S V_i$.

Let $I$ be the ideal of $\tau(E) \times_S V_i$ in $\ol{U} \times_S V_i$ and let $J$ be the ideal of the scheme theoretic image of $(U - E) \cap V_i \to \ol{U} \times_S V_i$. Now, the closure of the image of $(U - E) \cap V_i \to \ol{U} \times_S V_i$
has empty intersection with $\tau(E) \times_S V_i$ if and only if $I + J$ is the unit ideal of $\Gamma(\ol{U} \times_S V_i, \mathscr{O}_{\ol{U} \times_S V_i})$.
This holds if and only if there exist elements $g_j$ of the ideal of $\tau(E) \times_S V_i$ in $\ol{U} \times_S V_i$ and a cover $\{ U_j \}$ of $U - E \cap V_i$ so that $g_j$ restricts to a unit on $U_j$ for each $j$.

Since the image of $E$ agrees with the scheme theoretic image of $E$ in $\ol{U}$, which has ideal $I_{\tau(E)}$,
$\ol{U} - \tau(E) \cong U - E$ has a cover by the open sets $D(\ol{\lambda}(f))$ for $\phi(f, 0) \in I_{\tau(E)}$. Pulling back these functions to $\ol{U} \times_S V_i$ gives the result.
\end{proof}

\begin{proposition}
The morphism $\tau: C \to \ol{C}$ has $\ol{C}$ as its scheme theoretic image.
\end{proposition}
\begin{proof}
Since $\tau : C - E \to \ol{C} - \tau(E)$ is an isomorphism, this holds away from $E$. Therefore
it suffices to check that $\tau : U \to \ol{U}$ has $\ol{U}$ as its scheme-theoretic image, which is equivalent in turn to showing that $\ol{B}(U) \to \Gamma(U, \mathscr{O}_U)$ is injective.

To that end, suppose $\phi(f, c) \in \ol{B}(U)$ and $\ol{\lambda}(f) + \pi^*c = 0$. We want to show that $(f, c) = ((\rho - \ol{\lambda})(d), -\rho(d))$ for some $d \in \Gamma(S, \mathscr{O}_S(-\rho))$.

Working \'etale locally on $S$, we may find a section $\sigma : S \to C$ so that the image of $\sigma$ is contained in the locus where $\ol{\lambda} = \rho$. Set $d = \sigma^*(f) \in \Gamma(S, \mathscr{O}_S(-\rho))$. Our goal is now to show
that $(f,c) - ((\rho - \ol{\lambda})(\sigma^*(f)), -\rho(\sigma^*(f)))$ is zero. Consider the second coordinate first:
\begin{align*}
  c + \rho(\sigma^*(f)) = \sigma^{*}(\pi^*c + \ol{\lambda}(f)) = 0.
\end{align*}
It remains to show that $f - (\rho - \ol{\lambda})(\sigma^*(f)) = 0$. Note that on the complement of $E$ in $U$, $\rho - \ol{\lambda} = \rho$ and $\rho(\sigma^*f) = -\pi^*c = \ol{\lambda}(f) = f$. It follows that
the restriction of $f - (\rho - \ol{\lambda})(\sigma^*(f))$ to $U - E$ is zero, as desired. We also have that the difference is zero on the image of $\sigma$:
\[
  \sigma^*\big( f - (\rho - \ol{\lambda})(\sigma^*(f)) \big) = \sigma^*(f) - \sigma^*(f) = 0.
\]
Our strategy is now to localize enough that we can use these two properties to conclude that $f - (\rho - \ol{\lambda})(\sigma^*(f))$ is zero.

First, since $\Gamma(U, \mathscr{O}_C(-\ol{\lambda})) = \colim_{n \in \N} \Gamma(C, \mathscr{O}_C(-\ol{\lambda} + n\Sigma))$, where $\Sigma$ is the union of the divisors of the sections $\sigma_1, \ldots, \sigma_h$, there is
some $n$ so that $f - (\rho - \ol{\lambda})(\sigma^*(f))$ admits a lift $g$ to $\Gamma(C, \mathscr{O}_C(-\ol{\lambda} + n\Sigma))$. Since $\Sigma$ is a Cartier divisor, this lift is unique, so it is equivalent to show that $g$ is zero.
In addition, as in the proof of Proposition \ref{prop:no_coh_minus_lambda_bar}, we may choose $n$ so that $H^1(C, \mathscr{O}_C(-\ol{\lambda} + n\Sigma)) = 0$. In particular, by Theorem
\ref{thm:pushforward_flat_commutes_with_base_change}, taking global sections of $\mathscr{O}_C(-\ol{\lambda} + n\Sigma)$ commutes with arbitrary base change in $S$.

We now proceed by contradiction. Suppose that $g$ is nonzero. Then there is some point $s$ of $S$ so that the Zariski stalk of $g$ at $s$ is nonzero. Denote by $A_s$ the Zariski stalk of $A$ at $s$, by $\hat{A}$ the completion of $A_s$ at
$\mathfrak{m}_s$,  and let $C_s$ and $\hat{C}$ denote the respective pullbacks of $C$ to these rings. (Note that $C_s$ is not the fiber of $C$ over $s$, contrary to our earlier usage.)
Since $\Gamma(C_s, \mathscr{O}_{C_s}(-\ol{\lambda} + n\Sigma))$ is finitely generated and $A_s$ is Noetherian, \cite[Corollary 10.19]{atiyahmacdonald} implies that the map
\[
  \Gamma(C_s, \mathscr{O}_{C_s}(-\ol{\lambda} + n\Sigma)) \longrightarrow \Gamma(C_s, \mathscr{O}_{C_s}(-\ol{\lambda} + n\Sigma)) \otimes_{A_s} \hat{A} \cong \Gamma(\hat{C}, \mathscr{O}_{\hat{C}}(-\ol{\lambda} + n\Sigma))
\]
is injective: if $g_s$ is nonzero, its restriction to $\hat{A}$ must also be nonzero. This implies in turn that there is some integer $k$ so that the restriction of $g_s$ to $A_s / \mathfrak{m}_s^k$ is nonzero.

Now, $A_s/\mathfrak{m}_s^k$ is local Artinian, so admits a finite composition series $0 = M^{(r)} \subsetneq M^{(r - 1)} \subsetneq \cdots \subsetneq M^{(0)} =A_s/\mathfrak{m}_s^k$ with filtration quotients isomorphic to
$k(s) = A_s/\mathfrak{m}_s$. By flatness of $C$, we have an exact sequence
\[
  0 \to M^{(i+1)}\mathscr{O}_{\hat{C}}(-\ol{\lambda} + n\Sigma) \to M^{(i)}\mathscr{O}_{\hat{C}}(-\ol{\lambda} + n\Sigma) \to \mathscr{O}_{C \times_S k(s)}(-\ol{\lambda} + n\Sigma) \to 0
\]
for each $i$. Since $\mathscr{O}_C(-\ol{\lambda} + n\Sigma)$ is flat, commutes with base change, and has no $H^1$, we have an induced sequence
\begin{align*}
  0 \to M^{(i+1)}\Gamma(\hat{C}, \mathscr{O}_{\hat{C}}(-\ol{\lambda} + n\Sigma)) & \to M^{(i)}\Gamma(\hat{C}, \mathscr{O}_{\hat{C}}(-\ol{\lambda} + n\Sigma)) \\
      &\to \Gamma(C \times_S k(s), \mathscr{O}_{C \times_S k(s)}(-\ol{\lambda} + n\Sigma) ) \to 0
\end{align*}
Since $g$ is nonzero, there is a first $i$ so that $g \not\in M^{(i + 1)}\Gamma(\hat{C}, \mathscr{O}_{\hat{C}}(-\ol{\lambda} + n\Sigma))$. Consider the image of $g$ in $\Gamma(C \times_S k(s), \mathscr{O}_{C \times_S k(s)}(-\ol{\lambda} + n\Sigma))$. We may consider this as an element
of $\Gamma(U \times_S \ol{k(s)}, \mathscr{O}_{C \times_S \ol{k(s)}}(-\ol{\lambda}))$. Now the description of Proposition \ref{prop:sections_of_oc_minus_lambda_bar_on_fiber} applies. The component of $g$ in the $\Gamma(S, \mathscr{O}_S(-\rho))$
component must be zero since $g$ is zero on sigma. The part of $g$ in $\Gamma(Z, \mathscr{O}_Z(-\ol{\lambda}))$ must also be zero, since $g$ restricts to 0 on the complement of $E$ and $Z$ is reduced. This contradicts that $g$ is nonzero,
and we have the result. 
\end{proof}

\section{Appendix: Uniform sets of charts for log curves}

In this appendix we define and prove the existence of uniform sets of charts for log curves. Their purpose is to provide a cover of a log curve
$\pi : C \to S$ by standard local charts, nice enough that we can reason effectively about sections of the characteristic sheaf of $C$. We begin
by recalling a description of sections of the characteristic sheaf when $S$ is a geometric point.

\begin{proposition} \label{prop:pw_linear_fibers}
Let $\pi: C \to S$ be a log smooth curve over a geometric point. Then $\Gamma(C, \ol{M}_C)$ is in one-to-one correspondence with the set of piecewise linear functions on $\trop(C_{\ol{s}})$. If $\sigma \in \Gamma(C, \ol{M}_C)$, the
corresponding element of $PL(\trop(C_{\ol{s}}))$ is determined by taking $f_v$ to be the stalk of $\sigma$ at the generic point of $v$ for all $v \in V(\Gamma)$, and taking $n_h$ to be the slope of $\sigma$ at $h$ for all $h \in H(\Gamma)$.
\end{proposition}
\begin{proof}
This is \cite[Remark 7.3]{ccuw}.
\end{proof}

This description extends to a sufficiently small \'etale neighborhood in a nice way. When generizing a log curve, its nodes may smooth out in a way that we will see is controlled by the characteristic monoid of the base. Passing to tropicalizations,
generizing collapses edges and merges vertices. Following \cite{ccuw}, we now define a weighted edge contraction of a tropical curve, where a monoid homomorphism determines which edges are collapsed.

\begin{definition}
Let $\Gamma$, $\Gamma'$ be tropical curves with edge lengths in sharp monoids $P, P'$. A \emph{weighted edge contraction} $\pi : \Gamma \to \Gamma'$ consists of
\begin{enumerate}
  \item a surjective morphism of monoids $\pi_* : P \to P'$;
  \item a function $\pi : X(\Gamma) \to X(\Gamma')$
\end{enumerate}
so that
\begin{enumerate}
  \item $\pi \circ r_{\Gamma} = r_{\Gamma'} \circ \pi$;
  \item $\pi \circ i_{\Gamma} = i_{\Gamma'} \circ \pi$;
  \item $\pi$ restricts to a bijection $H(\Gamma) \to H(\Gamma')$;
  \item for each flag $f \in F(\Gamma')$, the preimage $\pi^{-1}(f)$ has exactly one element, which is necessarily a flag.
  \item for each vertex $v \in V(\Gamma')$, the tropical curve $\pi^{-1}(v)$ is connected with genus $g(v)$;
  \item for each edge $e \in E(\Gamma)$, $\pi(e)$ is a vertex of $\Gamma'$ if and only if $\pi_*(\delta(e)) = 0$;
  \item if $\pi(e') = e \in E(\Gamma')$, then $\pi_*(\delta(e')) = \delta(e)$.
\end{enumerate}

If $\pi : \Gamma \to \Gamma'$ is a weighted edge contraction with associated map on monoids $\pi_* : P \to P'$ and $f = ((f_v), (m_h)) \in PL(\Gamma)$, then $\pi_*f \in PL(\Gamma')$ is the piecewise linear
function whose value at a vertex $v'$ is equal to $\pi(f_v)$ for any $v$ mapping to $v'$ and whose slope at a half edge $h$ is $m_h$.
\end{definition}

\begin{remark}
The direction of the maps here is reversed from that in \cite{ccuw}.
\end{remark}

\begin{definition}
Let $\Gamma$ be a tropical curve. A \emph{path} $W$ in $\Gamma$ is a sequence $v_0e_1v_1e_2 \cdots e_kv_k$ of vertices and edges in $\Gamma$ so that the vertices $v_i$ are distinct and $v_{i - 1}$ and $v_i$ are the ends of the edge $e_i$ for all $i$.
The \emph{length} $\lambda(W)$ of $W$ is $\sum_{i = 1}^k \delta_{e_k}$.
\end{definition}

Given a tropical curve $\Gamma$ with edge lengths in $P$ and a surjective morphism $\pi_* : P \to P'$, there is an induced weighted edge contraction $\pi : \Gamma \to \Gamma'$ where
\begin{enumerate}
  \item $V(\Gamma') = V(\Gamma) / \sim$ where $v \sim w$ if and only if there is a path $W$ from $v$ to $w$ so that $\pi_*(\lambda(W)) = 0$;
  \item $H(\Gamma') = H(\Gamma)$;
  \item $E(\Gamma') = \{ e \in E(\Gamma) \mid \pi_*(\lambda(e)) \neq 0 \}$
\end{enumerate}
and the rest of the data are as expected.

\begin{definition} \cite[Definition 2.5.1]{ogus}
Let $\mathscr{F}$ be a sheaf of sets on a topological space $X$. A \emph{trivializing stratification} for $\mathscr{F}$ is a locally finite partition $\Sigma$ of $X$ such that:
\begin{enumerate}
  \item each element of $\Sigma$ is connected and locally closed;
  \item if $S_1$ and $S_2$ are elements of $\Sigma$ and $S_1 \cap \ol{S_2} \neq \emptyset$, then $S_1 \subseteq \ol{S_2}$.
  \item the restriction of $\mathscr{F}$ to each $S$ in $\Sigma$ is constant.
\end{enumerate}

A point $x$ of $X$ is a \emph{central point} for $\Sigma$ if $x$ belongs to the closure of every element of $\Sigma$.
\end{definition}

\begin{definition} \label{def:unif_charts}
Let $\pi : C \to S$ be a log curve and let $\ol{s}$ be a geometric point of $S$. A \emph{uniform set of charts for $C$ around $\ol{s}$} consists of
\begin{enumerate}
  \item A chart $S \overset{f}{\to} \mathbb{A}_{P}$ so that $f$ induces an isomorphism $P \cong \ol{M}_{S, \ol{s}}$.
  \item A finite set of smooth points $X$ of $C_{\ol{s}}$.
  \item A finite set of diagrams $D_i$ indexed by the set $I =E(\trop(C)) \sqcup H(\trop(C)) \sqcup X$, of the form
  \[
    D_i = \quad \vcenter{\vbox{\xymatrix{
    C \ar[d]_{\pi} & V_i \ar[l]_{\text{\'et}} \ar[d] \ar[r]^{\text{\'et}} & V_i' \ar[r] \ar[d] & \mathbb{A}_{Q_i} \ar[d] \\
    S \ar@{=}[r] & S \ar@{=}[r] & S \ar[r]^f & \mathbb{A}_P,
  }}}
  \]
  where, in the terminology of Theorem \ref{thm:log_curve_local_structure}, $D_i$ is a germ of a smooth point around $i$ for $i \in X$, a germ of a marked point around $i$ for $i \in H(\trop(C))$, and the germ of a node around $i$ for $i \in E(\trop(C))$.
\end{enumerate}

such that
\begin{enumerate}[(A)]
  \item $S$ is admits a trivializing stratification with central point $s$;
  \item $\{ V_i \to C \}_{i \in I}$ is an \'etale cover of $C$;
  \item $\Gamma(C, \ol{M}_C) \to \Gamma(C_{\ol{s}}, \ol{M}_{C_{\ol{s}}}) \cong PL(\trop(C_{\ol{s}}))$ is an isomorphism.
  \item If $\ol{t}$ is any geometric point of $S$, the map
   \[
     PL(\trop(C_{\ol{s}})) \overset{\sim}{\to} \Gamma(C, \ol{M}_C) \to \Gamma(C_{\ol{t}}, \ol{M}_{C_{\ol{t}}}) \overset{\sim}{\to} PL(\trop(C_{\ol{t}}))
   \]
   coincides with the map $f \mapsto \phi_*f$ where $\phi : \trop(C_{\ol{s}}) \to \trop(C_{\ol{t}})$ is the weighted edge contraction associated to the map $\phi_* : \ol{M}_{S, \ol{s}} \cong P \cong \Gamma(S, \ol{M}_S) \to \ol{M}_{S, \ol{t}}$.
\end{enumerate}
\end{definition}

Uniform sets of charts are a modification of the atomic neighborhoods considered by Santos-Parker in \cite{keli_thesis}.

\begin{lemma} \label{lem:families_stay_connected}
Suppose $\pi : C \to S$ is a proper morphism where $S$ is a DVR with special point $s$ and generic point $\eta$, so that $C_s$ is nonempty and $C_\eta$ is geometrically reduced, geometrically connected, and dense in $C$. Then $C_{\ol{s}}$ is connected.
\end{lemma}
\begin{proof}
Let $B = \Gamma(C, \mathscr{O}_{C})$ and let $T = \Spec B$. Consider the Stein factorization
\[
  \xymatrix{
    C \ar[dr]_{\pi} \ar[rr]^f & & T \ar[dl]^g \\
     & S &
  }
\]
Recall that $f$ is proper with geometrically connected fibers and $g$ is finite \cite[Tag 03H0]{stacks-project}. Since $f$ is also scheme-theoretically surjective, $f$ is surjective.
Our hypotheses ensure that $T \to S$ is an isomorphism over $\eta$; write $\xi$ for the copy of the generic point of $S$ in $T$.
Note that $\xi = f(C_\eta)$ is dense in $T$: if it were not, then $f^{-1}(\ol{f(C_\eta)})$ would be a closed set between $C_\eta$ and $C$.

Let $p$ be a point in the fiber of $T$ over $s$. Since $\xi$ is dense in $T$, $p$ generizes to $\xi$. Write $\pp$ for the prime of $B$ corresponding to $p$ and $\qq$ for the prime corresponding to $\xi$.
Note that $\qq$ is the unique minimal prime of $B$, so $B/\qq$ is integral. Then the ring $(B/\qq)_{\pp}$ is a local integral domain of dimension 1, and the natural map $\Spec (B/\qq)_\pp \to T$ restricts to an isomorphism over $\xi$.
The composite map $\phi : A \to B \to (B/\qq)_{\pp}$ is injective and $\Spec (B/\qq)_{\pp} \to \Spec A$ is an isomorphism on generic points, so we have a chain of inclusions $A \overset{\phi}{\to} (B/\qq)_{\pp} \to \mathrm{Frac}(A)$. Since valuation rings are the
maximal local rings with respect to inclusion in their fields of fractions, we conclude that $\phi$ is an isomorphism. Since $B / \qq \to (B/\qq)_{\pp}$ is injective, and $A \to B / \qq \to (B/\qq)_\pp$ is surjective, $A \to B / \qq$ is also an isomorphism.

Then $T \to S$ may be identified with a retraction of $T^{red} \to T$, a universal homeomorphism. Since geometric fibers of $f$ are connected, the geometric fibers of $\pi$ are connected, and we are done.
\end{proof}

\begin{lemma} \label{lem:charts_for_dual_graph_DVR}
Let $\pi : C \to S$ be a log curve where $S$ is the spectrum of a DVR with special point $s$ and generic point $\eta$. Suppose that the data of a uniform set of charts about $\ol{s}$ are given for $\pi$ satisfying (A), (B), and (C). Then (D) holds.
\end{lemma}

\begin{proof}
There are two things to check. The first is that the tropicalization of the geometric generic fiber is the weighted edge contraction induced by $\phi: \ol{M}_{S,s} \to \ol{M}_{S, \ol{\eta}}$. The second is that the map on piecewise linear functions is the pushforward as claimed.
Both of these may be checked on a finite \'etale cover of $S$, so we may assume that the components of $C$ are in bijection with the components of $C_{\ol{\eta}}$ via the map $C_{\ol{\eta}} \to C_{\eta}$ followed by taking the closure.

Consider the charts
 \[
    \xymatrix{
    C \ar[d]_{\pi} & V_e \ar[l]_{\text{\'et}} \ar[d] \ar[r]^{\text{\'et}} & V_e' \ar[r] \ar[d] & \mathbb{A}_{Q_e} \ar[d] \\
    S \ar@{=}[r] & S \ar@{=}[r] & S \ar[r]^f & \mathbb{A}_P
  }
  \]
for $e \in \trop(C_{\ol{s}})$. Recall that $V_e' = \Spec \mathscr{O}_S[x,y]/(xy - t)$ where $t = \epsilon(\tilde{\delta_e})$ for a lift $\tilde{\delta_e}$ of $\delta_e$ to $\Gamma(S, M_S)$. Taking the fiber over $\ol{\eta}$, we see that $V_e'$ has a node
over $\ol{\eta}$ $\iff$ the restriction of $t$ to $\ol{\eta}$ is not a unit $\iff$ $\delta_e$ restricts to something nonzero in $\ol{M}_{S, \ol{\eta}}$. Therefore, the edges of $\trop(C_{\ol{\eta}})$ are in bijection with the edges of $\trop(C_{\ol{s}})$
whose lengths remain nonzero under $\phi_*$.

We have that the edges of $\trop(C_{\ol{\eta}})$ are the desired ones; next we check that we have the desired components. We have that the components of $C$ are in bijection with the irreducible components of $C_{\ol{\eta}}$. Intersecting
with the fiber over $s$, the irreducible components of $C$ induce a partition of the irreducible components of $C_s$ with a component $Z_i$ of $C_s$ belonging to the subset indexed by the component $Z$ of $C$ if and only if $Z_i \subseteq Z$. Pulling
back to $C_{\ol{s}}$, we get a partition of the components of $C_{\ol{s}}$ indexed by the components of $C_{\ol{\eta}}$. Write $Z \rightsquigarrow Z_i$ where $Z$ is a component of $C_{\ol{\eta}}$ if and only if the component $Z_i$ of $C_{\ol{s}}$ belongs to the part of the partition indexed by $Z$.

I claim that $Z \rightsquigarrow Z_1$ and $Z \rightsquigarrow Z_2$ if and only if there is a path $W$ from $Z_1$ to $Z_2$ in $\trop(C_{\ol{s}})$ so that $\phi_*(\lambda(W)) = 0$. The ``if'' direction is clear: for each edge $e$ in the path $W$,
the fiber of $V_e'$ over $\ol{\eta}$ will have just one irreducible component, implying that the two branches of $e$ smooth to a single component over $\ol{\eta}$. Conversely, suppose that there is no such path between $Z_1$ and $Z_2$.
We may then choose an edge cut $\{e_1, \ldots, e_r \}$ separating $Z_1$ from $Z_2$ in $\trop(C_{\ol{s}})$ so that $\phi_*(\delta_{e_i}) \neq 0$ for all $i$. Let $\tilde{Z} \to S$ be the family obtained by blowup up $Z$ along the closed subscheme of nodes
corresponding to the edges $e_1, \ldots, e_r$. Then the proper transforms of $Z_1$ and $Z_2$ belong to separate connected components of the geometric fiber of $\tilde{Z}$ over $s$, while the geometric fiber over $\eta$ is connected. This
contradicts Lemma \ref{lem:families_stay_connected}. We conclude that $\trop(C_{\ol{\eta}})$ is obtained by contracting the edges of $\trop(C_{\ol{s}})$ whose lengths are sent to 0 by $\phi_*: \ol{M}_{S, \ol{s}} \to \ol{M}_{S, \ol{t}}$,
as required.

Now, global sections of $\ol{M}_{C}$ are in bijection with tuples of sections $(\sigma_i)_{i \in I}$ where $\sigma_i \in \Gamma(V_i, \ol{M}_{C}) \cong Q_i$ so that whenever $\ol{t}$ is a geometric point of $C$ factoring through $V_i$ and $V_j$,
the stalks $(\sigma_i)_{\ol{t}}$ and $(\sigma_j)_{\ol{t}}$ agree. The bijection $\Gamma(C, \ol{M}_{C}) \cong PL(\trop(C_{\ol{s}}))$ is given by sending $(\sigma_i)_{i \in I}$ to the piecewise linear function $((f_v), (m_h))$ where
\begin{enumerate}
  \item for each $v \in V(\trop(C_{\ol{s}}))$, and any $i$ so that $V_i \to C$ contains $v$ in its image, $f_v$ is the stalk of $\sigma_i$ at the generic point of the component $v$, and
  \item for each $h \in H(\trop(C_{\ol{s}}))$, the slope $m_h$ is the second component of $\sigma_h \in Q_h = P \oplus \N$.
\end{enumerate}
Note that the pullback of the charts $V_i$ to $\ol{\eta}$ preserves that they are a uniform set of charts except possibly for property (D). The map
$\Gamma(C, \ol{M}_C) \to \Gamma(C_{\ol{\eta}}, \ol{M}_{C_{\ol{\eta}}}) \cong PL(\trop(C_{\ol{\eta}}))$ is then given by restricting to the fiber over $\ol{\eta}$, then taking stalks of the $\sigma_i$ at appropriate points. We know that
taking these stalks at smooth points agrees with the map $\phi_* : \ol{M}_{S, \ol{s}} \to \ol{M}_{S, \ol{\eta}}$, so we get the result.
\end{proof}

\begin{proposition}
Let $\pi : C \to S$ be a log curve and let $\ol{s}$ be a geometric point of $S$. Then there is an \'etale neighborhood $U \to S$ of $\ol{s}$ admitting a uniform set of charts around a geometric point over $\ol{s}$.
\end{proposition}
\begin{proof}
Our strategy will be to gradually ensure more of the definition holds by working steadily more locally.

By standard arguments, we may assume that $S$ is Noetherian and finite dimensional.

Next we gather the data (i)-(iii).
For each edge $e$ of $\trop(C_{\ol{s}})$ (that is, each node of $C_{\ol{s}}$), use Theorem \ref{thm:log_curve_local_structure} to choose a diagram
\[
  \xymatrix{
    C \ar[d]_{\pi} & V_e \ar[l]_{\text{\'et}} \ar[d] \ar[r]^{\text{\'et}} & V_e' \ar[r] \ar[d] & \mathbb{A}_{Q_e} \ar[d] \\
    S  & \ar[l] U_e \ar@{=}[r] & U_e \ar[r]^f & \mathbb{A}_P
  }
\]
where $V_e \to C$ is an \'etale neighborhood of $e$.
Similarly, for each half edge $h$ of $\trop(C_{\ol{s}})$ (that is, each marked point of $C_{\ol{s}}$, use Theorem \ref{thm:log_curve_local_structure} to choose a diagram
\[
  \xymatrix{
    C \ar[d]_{\pi} & V_h \ar[l]_{\text{\'et}} \ar[d] \ar[r]^{\text{\'et}} & V_h' \ar[r] \ar[d] & \mathbb{A}_{Q_h} \ar[d] \\
    S  & \ar[l] U_h \ar@{=}[r] & U_h \ar[r]^f & \mathbb{A}_P
  }
\]
where $V_h \to C$ is an \'etale neighborhood of $h$ and $Q_h = P \oplus \N\gamma_h$. Finally, for each smooth geometric point $p$ of $C_{\ol{s}}$, we may choose a diagram
\[
  \xymatrix{
    C \ar[d]_{\pi} & V_p \ar[l]_{\text{\'et}} \ar[d] \ar[r]^{\text{\'et}} & V_p' \ar[r] \ar[d] & \mathbb{A}_{Q_p} \ar[d] \\
    S  & \ar[l] U_p \ar@{=}[r] & U_p \ar[r]^f & \mathbb{A}_P
  }
\]
where $V_p \to C$ is an \'etale neighborhood of $p$ and $Q_p = P$.

By quasicompactness, we may choose a finite set $X$ of these points $p$, so that $\{ V_e \to C \}_{e \in E(\Gamma)} \cup \{ V_h \to C \}_{h \in H(\Gamma)} \cup \{ V_p \to C \}_{p \in X}$ jointly cover $C_{\ol{s}}$. Let $I = E(\Gamma) \cup H(\Gamma) \cup X$.

Taking the fiber product over all the $U_i$s, then choosing a connected component, we may assume that all of the $U_i$s are equal to a single connected \'etale neighborhood $U$ of $\ol{s}$. Let $Z$ be the complement of the
union of the images of the $V_i$s in $C$. Since $Z$ is closed and $\pi$ is proper, $\pi(Z)$ is closed and does not contain $\ol{s}$. Replacing $U$ with the complement of $(U \to S)^{-1}(\pi(Z))$ in $U$, we may assume that the $V_i$s cover $C|_U$. We now have the data (i)-(iii) and conditions (A) and (B).

Now, by Proposition \ref{prop:pw_linear_fibers}, we have that $PL(\trop(C_{\ol{s}})) \cong \Gamma(C_{\ol{s}}, \ol{M}_{C_{\ol{s}}})$, which is isomorphic in turn with the stalk $(\pi_*\ol{M}_{C})_{\ol{s}}$.

For any \'etale neighborhood $V$ of $\ol{s}$ in $U$, $\Gamma(V \times_S C, \ol{M}_{C})$ is a sharp fine monoid, since it is a submonoid of the product of the sharp fine monoids $\Gamma(V \times_U V_i, \ol{M}_C)$. Therefore we can express
the stalk $(\pi_*\ol{M}_C)_{\ol{s}}$ as a colimit of fine monoids of the form $\Gamma(V \times_S C, \ol{M}_C)$ where $V$ is an \'etale neighborhood of $\ol{s}$.
The piecewise linear functions on $\trop(C_{\ol{s}})$ form a finitely generated monoid, so we can find an \'etale neighborhood $V$ so that $\Gamma(V \times_S C, \ol{M}_C) \to (\pi_*\ol{M}_C)_{\ol{s}}$ is surjective.
Since $\Gamma(V \times_S C, \ol{M}_C)$ is sharp and finitely generated, it is Noetherian \cite[Corollary 2.1.10]{ogus}, so the congruence relation on $\Gamma(V \times_S C, \ol{M}_C)$ induced by the map
$\Gamma(V \times_S C, \ol{M}_C) \to (\pi_*\ol{M}_C)_{\ol{s}}$ is finitely generated. It follows that there is an \'etale neighborhood $W$ of $\ol{s}$ so that $\Gamma(W \times_S C, \ol{M}_C) \cong (\pi_*\ol{M}_C)_{\ol{s}}$
is an isomorphism. Replacing $U$ with $W$ and pulling back the neighborhoods $V_i$ ensures that ($C$) holds.

Now by \cite[Theorem 2.5.4]{ogus}, we may replace $U$ with a smaller \'etale neighborhood on which $\ol{M}_{S}$ is has a trivializing stratification and $s$ is a central point. Let $\ol{t}$ be any geometric point of $U$ with image $t$.
There is a generization $s \leftsquigarrow \eta$ and a specialization $\eta \to t$ where $\eta$ is the generic point of the irreducible component of the stratum containing $t$. We can extend $s \leftsquigarrow \eta$ to a finite sequence of generizations $s = s_0 \leftsquigarrow s_1 \leftsquigarrow \cdots \leftsquigarrow s_k = \eta$ and $\eta \rightsquigarrow t$ to a sequence of specializations $\eta = t_l \rightsquigarrow \cdots \rightsquigarrow t_0 = t$ where each generization and specialization
has codimension one. We can find a sequence of DVRs $A_i$ and maps $\Spec A_i \to U$ with generic point hitting $s_{i}$ and special point hitting $s_{i - 1}$. Similarly, there is a sequence of DVRs $B_i$ and maps $\Spec B_i \to U$ with generic point
hitting $t_i$ and special point hitting $t_{i - 1}$. Note that the map $\ol{M}_{t_{i-1}} \to \ol{M}_{t_i}$ is an isomorphism for all $i$. Pulling back to both sequence of DVRs and applying Lemma \ref{lem:charts_for_dual_graph_DVR} yields that (D) holds.
\end{proof}

\bibliographystyle{amsalpha}
\bibliography{contraction}

\end{document}